\pdfoutput=1
\documentclass{article}
\usepackage{iclr2027_arxiv,times}

\usepackage{amsmath,amssymb,amsthm,mathtools}
\usepackage{enumitem}
\usepackage{algorithm}
\usepackage{algpseudocode}
\algrenewcommand\algorithmicrequire{\textbf{Input:}}
\algrenewcommand\algorithmicensure{\textbf{Output:}}
\usepackage{graphicx}
\usepackage{float}
\usepackage{placeins}
\usepackage{booktabs,array}
\usepackage{multirow}
\usepackage{wrapfig}
\usepackage{microtype}
\usepackage{xurl}
\usepackage{hyperref}
\usepackage[font=small,labelfont=bf]{caption}
\usepackage{subcaption}

\hypersetup{
    colorlinks=true,
    citecolor=blue,
    linkcolor=blue,
    urlcolor=blue
}

\graphicspath{{figures/}}

\newcommand{\R}{\mathbb{R}}

\newcommand{\E}{\mathbb{E}}

\newcommand{\sign}{\operatorname{sign}}

\newcommand{\clamp}{\operatorname{clamp}}
\newcommand{\dd}{\,\mathrm{d}}

\newcommand{\one}{\mathbf{1}}
\newcommand{\norm}[1]{\lVert #1\rVert}
\newcommand{\ip}[2]{\left\langle #1,#2\right\rangle}

\theoremstyle{definition}

\theoremstyle{plain}
\newtheorem{theorem}{Theorem}
\newtheorem{lemma}[theorem]{Lemma}
\newtheorem{corollary}[theorem]{Corollary}
\newtheorem{proposition}[theorem]{Proposition}
\newtheorem{assumption}{Assumption}
\newtheorem{model}{Model}

\theoremstyle{remark}
\newtheorem{remark}{Remark}

\title{Beyond Shadow Weights: Quantization-Aware Training as Quantized-Endpoint Descent}

\newcommand{\authmark}[1]{\textsuperscript{\normalfont\mdseries #1}}
\author{%
Sheng-An Xu\authmark{1} \quad
Hanyang Li\authmark{1} \quad
Jianhao Ma\authmark{2} \quad
Ying Cui\authmark{1}\\[0.5ex]
{\normalfont\footnotesize \authmark{1}Department of Industrial Engineering and Operations Research, University of California, Berkeley}\\
{\normalfont\footnotesize \authmark{2}Department of Industrial Engineering, Tsinghua University}
}

\begin{document}

\maketitle

\begin{abstract}
Quantization-aware training (QAT) updates a full-precision shadow weight $\mathbf{x}$ but deploys the quantized endpoint $Q(\mathbf{x})$.  Existing explanations for QAT largely view its success through the lens of shadow weights: QAT can move $\mathbf{x}$ toward flatter basins, gain robustness from quantization-induced oscillations, or balance the shadow loss $f(\mathbf{x})$ against the quantization error $\|\mathbf{x}-Q(\mathbf{x})\|_2$.  These perspectives do not directly explain the empirical observation that the deployed endpoint loss $f(Q(\mathbf{x}))$ improves while the shadow loss $f(\mathbf{x})$ does not, and can even increase substantially. In this paper, we offer a different explanation by treating QAT as finite-grid endpoint dynamics. Motivated by the approximate normality of rescaled pretrained weights, we propose an idealized model for the residual phase, which records where each shadow weight sits inside its quantization cell as a fraction of the cell width. This model leads to a crossing law that determines which coordinates cross quantization boundaries after a shadow update. Inspired by the idealized model and signal-imbalance phenomenon in QAT, we further propose \emph{QAR (Quantization with Amplified Routing)}, an algorithmic framework that directly operates on the quantization code. In contrast to QAT, QAR is both theoretically grounded and memory-efficient: it admits feasible-gradient bounds for a family of power amplifiers up to unavoidable finite-grid floors without retaining a full-precision shadow weight copy. Experiments on post-training of large language models provide evidence consistent with the endpoint view and show that QAR can be comparable to or better than QAT with smaller memory cost.
\end{abstract}

\section{Introduction}
\label{sec:introduction}

Modern deep neural networks have achieved broad empirical success across language, vision, and decision-making tasks \citep{brown2020language,he2016deep,silver2016mastering}, by scaling parameters, data, and compute \citep{kaplan2020scaling,hoffmann2022training}. But this scaling comes at a cost. As the model becomes larger, deployment is increasingly constrained by memory footprint, and inference slows down due to high memory bandwidth and compute throughput. This phenomenon is highly pronounced for large language models (LLMs).  Low-bit quantization addresses this issue by replacing high-precision weights with a small set of representable values.  Post-training quantization methods compress a pretrained model with little additional training, and they have become a central tool for LLM deployment. Representative examples include GPTQ~\citep{frantar2022gptq}, SmoothQuant~\citep{xiao2023smoothquant}, and AWQ~\citep{lin2024awq}.  However, at more aggressive bit widths, the quantized model can lose accuracy substantially. In this regime, quantization-aware training (QAT) becomes the standard method for repair. For example, LLM-QAT~\citep{liu2023llmqat} and EfficientQAT~\citep{chen2025efficientqat} make this repair step practical for LLMs with data-free distillation or reduced memory cost.

Conceptually, QAT tries to solve the deployment problem. Let $f:\R^d\to\R$ be the full-precision loss, $Q_b$ be a $b$-bit quantizer, and $\mathcal Q_b$ be the finite set of deployed endpoints. To achieve the best deployment performance, the training objective should be a discrete constrained problem
\begin{equation}\label{eq:finite end}
    \min_{\mathbf{q}\in\mathcal Q_b} f(\mathbf{q}).
\end{equation}
By introducing a full-precision shadow weight $\mathbf{x}$, we can equivalently write the above problem as
\[
    \min_{\mathbf{x}\in\R^d}\; F_b(\mathbf{x})=f(Q_b(\mathbf{x})).
\]
Since $Q_b$ is constant on each quantization cell, the exact gradient of $f(Q_b(\mathbf{x}))$ is zero almost everywhere; direct differentiation therefore gives no useful descent direction. QAT bypasses this degeneracy through the straight-through estimator (STE), which treats $Q_b$ as the identity in the backward pass~\citep{bengio2013estimating}.  In the notation above, one step of QAT has the form
\[
    \mathbf{x}_{t+1}=\mathbf{x}_t-\eta_t \mathbf{u}_t,
    \qquad
    \mathbf{u}_t=\mathcal U_t(\nabla f(Q_b(\mathbf{x}_t))),
\]
where $\mathcal{U}_t(\cdot)$ is a time-dependent update map. STE-based QAT has been effective in practice~\citep{courbariaux2015binaryconnect,zhou2016dorefa,jacob2018quantization,choi2018pact,esser2020learned}, but it remains largely heuristic for solving the problem \eqref{eq:finite end}. From a classical optimization viewpoint, QAT is unusual because the variable being updated is not the variable being deployed. During training, QAT maintains a full-precision shadow weight $\mathbf{x}$, also called the latent weight, but at deployment the model uses the quantized endpoint $\mathbf{q}$. Thus, to analyze QAT, {\bf the central analytical question is not how the shadow trajectory behaves, but how the deployed endpoint trajectory evolves.} 

The main analytical difficulty lies in the discreteness induced by quantization. Standard first-order arguments alone fail to match the motion of the deployed endpoint. By $L$-smoothness of $f$, the classical gradient update $\mathbf{x}_{t+1}=\mathbf{x}_t-\eta\nabla f(\mathbf{x}_t)$ yields $f(\mathbf{x}_{t+1})-f(\mathbf{x_t}) \le -\eta\big(1-\tfrac{L\eta}{2}\big)\norm{\nabla f(\mathbf{x_t})}_2^2$, so choosing $\eta$ small enough guarantees a descent in the objective value. However, the loss at the QAT endpoint $f(Q_b(\mathbf{x}_t))$ behaves differently. Although the change in shadow weights $\mathbf{x}_{t+1} - \mathbf{x}_t$ is proportional to the stepsize $\eta_t$, once $\mathbf{x}_{t+1}$ crosses a quantization boundary, the change in the deployed endpoint $Q_b(\mathbf{x}_{t+1}) - Q_b(\mathbf{x}_t)$ is a grid jump. Shrinking $\eta$ cannot reduce the size of this jump, hence descent of the endpoint loss no longer holds.

\paragraph{Existing Explanations.} For QAT with weight-only quantization,  explanations in the existing literature can be grouped into three categories. The first is a flatness view: QAT can drive the shadow weight into a flat region where replacing $\mathbf{x}$ by $Q_b(\mathbf{x})$ changes the loss only mildly. For example, \citet{javed2025qtdog} model quantization as adding noise for the shadow weight. In this way, QAT has an implicit regularizer for the second-order term and can steer the shadow point toward flatter regions. This perspective is inspired by the broader literature which links flatter regions of the loss landscape to better generalization~\citep{hochreiter1997flat,keskar2017largebatch,jiang2020fantastic,foret2021sharpness}. A related explanation builds on the river-valley picture of loss landscapes \citep{wen2024river}. For instance, \citet{li2026understandingqat} argue that when the deployed endpoint leaves a low-loss basin, the STE update can include an inward component that pulls the quantized endpoint back toward the basin, where $f(Q_b(\mathbf{x}))$ stays close to $f(\mathbf{x})$. The second view interprets QAT as minimizing the shadow loss $f(\mathbf{x})$ plus a regularizer that penalizes distance of the shadow weight to the quantized set. ProxQuant~\citep{bai2019proxquant} proposes a continuous regularized problem that encourages the shadow weight to approach the quantized set, and STE-based QAT corresponds to the hard-projection limiting case of the proximal framework.
The third view emphasizes oscillation robustness gained from quantization. In a toy model, \citet{wenshoj2025oscillations} argue that the STE update contains a residual-driven component that pushes shadow weights away from their nearest quantization level and toward cell boundaries. And this effect can be achieved by adding a regularization term to the shadow loss $f(\mathbf{x})$.

\begin{figure}[t]
\vspace{-1.2em}
\centering
\begin{subfigure}[t]{0.28\linewidth}
    \centering
    \includegraphics[width=\linewidth]{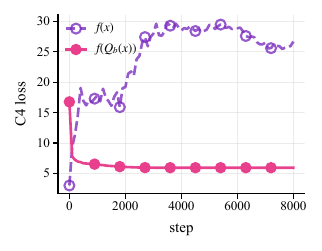}
\caption{2-bit evaluation loss}
\end{subfigure}
\begin{subfigure}[t]{0.32\linewidth}
    \centering
    \includegraphics[width=\linewidth]{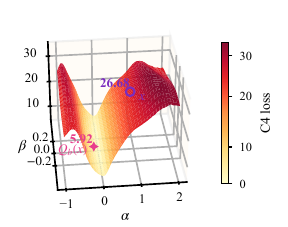}
\caption{Loss landscape at the final step}
\end{subfigure}
\begin{subfigure}[t]{0.30\linewidth}
    \centering
    \includegraphics[width=\linewidth]{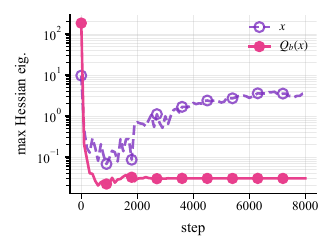}
\caption{2-bit max Hessian eigenvalue}
\end{subfigure}
\caption{Qwen-2.5-0.5B 2-bit loss and curvature diagnostics using seed 0. (a) contrasts the increasing shadow loss with the low endpoint loss. (b) plots evaluation loss on $Q_b(\mathbf{x})+\alpha(\mathbf{x}-Q_b(\mathbf{x}))+\beta \norm{\mathbf{x}-Q_b(\mathbf{x})}_2\mathbf v$ for a random unit $\mathbf v\perp(\mathbf{x}-Q_b(\mathbf{x}))$. (c) shows curvature better aligned with the deployed endpoint in 2 bits.}
\label{fig:qwen25-loss-diagnostics}
\vspace{-1.2em}
\end{figure}

Figure~\ref{fig:qwen25-loss-diagnostics} illustrates a regime that is not described by any account requiring the shadow loss $f(\mathbf{x}_t)$ to decrease.  In this 2-bit Qwen-2.5-0.5B run, the loss at the shadow weight $f(\mathbf{x}_t)$ increases substantially while the deployed endpoint loss $f(Q_b(\mathbf{x}_t))$ remains much lower. Moreover, the loss landscape around the deployed endpoint $Q_b(\mathbf{x}_t)$ appears to be flatter, rather than around $\mathbf{x}_t$. These observations suggest a stronger conclusion: QAT can improve the deployed quantized endpoint, even at the expense of sacrificing the performance of the full-precision objective. This motivates us to analyze the deployed endpoint loss directly.

\paragraph{Contribution.} Our contributions are two-fold. First, we analyze QAT through the residual $\mathbf{r}_t=\mathbf{x}_t-Q_b(\mathbf{x}_t)$, which records the position of the shadow point inside its quantization cell. Motivated by the approximate normality of the rescaled pretrained weights, we introduce an idealized residual-phase model that explains why pooled phases can approach uniformity at higher bit widths. This model expresses endpoint motion through random boundary-crossing indicators, allowing the dynamics to be analyzed directly at the quantized endpoint without explicitly tracking the shadow weights. It offers an approximate perspective on when QAT can help training: the bound certifies expected endpoint descent when the first-order signal exceeds the upper bound on the finite-jump penalty.

Second, we identify asymmetric oscillations caused by signal imbalance in QAT and use this mechanism to develop \textbf{QAR} (\textbf{Q}uantization with \textbf{A}mplified \textbf{R}outing), a memory-efficient endpoint-training framework that updates quantized codes directly and prioritizes transitions associated with stronger signals. We establish convergence guarantee for QAR and characterize when signal amplification remains beneficial under stochastic gradient noise. Post-training experiments on Qwen-2.5-0.5B and Llama-3.2-3B support finite-grid endpoint view and show that QAR can achieve performance comparable to or better than QAT with up to $\textbf{26.3\%}$ lower peak training memory.

\section{Preliminaries}
\label{sec:preliminaries}

To begin with, quantization maps high-precision floating-point values to low-precision discrete representations, thereby reducing memory footprint.  We use the standard uniform affine quantization notation \citep{nagel2021whitepaper}. Throughout the paper, unless stated otherwise, we use a signed symmetric uniform $b$-bit quantizer $Q_b(\cdot)$.\footnote{In this convention, we drop the most negative code so that the positive and negative endpoints have equal magnitude. The resulting codebook has $2^b-1$ levels. For example, the $b=2$ setting in this paper has levels $\{-1,0,1\}$, so it corresponds to the ternary setting, or $1.58$-bit, in common terminology.}  Let $\mathbf{s}=(s_1,s_2,\dots,s_d)>0$ be the fixed local scale. Define $K_b=2^{b-1}-1$ and $\mathcal K_b=\{-K_b,-K_b+1,\ldots,K_b\}$. For shadow weight $\mathbf{x}$, the i-th coordinate of the deployed endpoint is
\[
    Q_b(\mathbf{x})_i
    =
    s_i\,
    \clamp\!\left(
        \left\lfloor \frac{x_i}{s_i}\right\rceil;\,-K_b,K_b
    \right),
\]
where $\lfloor\cdot\rceil$ denotes rounding to the nearest integer and $\clamp(a;\ell,u)=\min\{\max\{a,\ell\},u\}$. When the bit width is clear from context, we write $Q(\mathbf{x})$ for $Q_b(\mathbf{x})$ for simplicity. Then the deployed values in coordinate $i$ are $s_i\mathcal K_b$, with adjacent interior spacing $s_i$.  For an interior code $k\in\{-K_b+1,\ldots,K_b-1\}$, the non-saturated cell is $C_{i,k} = \left[ s_i\bigl(k-\frac1 2\bigr), s_i\bigl(k+\frac1 2\bigr) \right)$. We call coordinate $i$ non-saturated when its deployed code $k_i=Q(\mathbf{x})_i/s_i$ satisfies $-K_b<k_i<K_b$; coordinates with $k_i=\pm K_b$ are saturated.

Let $\mathbf{g}=\nabla f(\mathbf{q})$ be the endpoint gradient.  A one-step QAT shadow update with stepsize $\eta$ is
\begin{equation}
\label{eq:1step_QAT}
    \mathbf{x}^+=\mathbf{x}-\eta \mathbf{u}
    \quad
    \text{with }\mathbf{u}=\mathcal U_t(\mathbf{g}).
\end{equation}
Here $\mathcal U_t(\cdot)$ is the update rule applied to the endpoint gradient at step $t$. For example, GD takes $\mathcal{U}_t(\mathbf{g})=\mathbf{g}$ and signGD corresponds to $\mathcal{U}_t(\mathbf{g})=\sign(\mathbf{g})$. We further write $\mathbf{q}^+=Q(\mathbf{x}^+)$, and call coordinate $i$ active in this step when $u_i\ne0$.  Throughout, we use the convention $\sign(0)=0$.

\begin{assumption}[Bounded update direction]
\label{ass:update-direction}
There exists a constant $U <\infty$ such that the update direction $\mathbf{u}=\mathcal U_t(\mathbf{g})$ satisfying $\norm{\mathbf{u}}_\infty\le U$.
\end{assumption}

\begin{assumption}[Smoothness]
\label{ass:smooth}
The loss $f$ is differentiable, and there exists a nonnegative vector $\mathbf{L}=(L_1,\ldots,L_d)$ such that for all $\mathbf{x},\mathbf{y}$,
\[
    f(\mathbf{y})
    \le
    f(\mathbf{x})+\ip{\nabla f(\mathbf{x})}{\mathbf{y}-\mathbf{x}}
    +
    \frac12\sum_{i=1}^d L_i(y_i-x_i)^2 .
\]
This coordinatewise smoothness is commonly used in sign-gradient analyses such as~\citet{bernstein2018signsgd}.  Standard $L$-smoothness implies this assumption by taking $L_i=L$ for all coordinates.
\end{assumption}

\begin{lemma}[First-order descent alignment]
\label{lem:coordinatewise-crossing-sign}
For the one-step QAT update \eqref{eq:1step_QAT} at point $\mathbf{x}$ with stepsize $\eta$, we have $\mathbf{u} \odot (\mathbf{q}^+ - \mathbf{q}) \le0.$ If in addition ${\mathbf{g}}\odot {\mathbf{u}}\ge0$, then $\mathbf{g} \odot (\mathbf{q}^+ - \mathbf{q}) \le 0$, and consequently,
\[
    \ip{\mathbf{g}}{\mathbf{q}^+-\mathbf{q}}\le0.
\]
Here $\odot$ denotes the Hadamard product, and the inequalities should be understood coordinatewise.
\end{lemma}

The proof is deferred to Appendix~\ref{app:proof-lemma-alignment}. Lemma~\ref{lem:coordinatewise-crossing-sign} shows that each realized endpoint jump moves opposite to the update direction.  If the update is coordinatewise aligned with the endpoint gradient, the jump is also aligned with first-order descent.  However, to get the descent of the loss, we need a more fine-grained bound. To achieve the bound, we assume that the stepsize is small enough that $\eta |u_i|<s_i$ for every coordinate $i$. In this case, each coordinate crosses at most one interior threshold per step, so $(\mathbf{q}^+-\mathbf{q})_i\in\{-s_i,0,s_i\}$. Let $\mathbf{J}=(J_1,\ldots,J_d)$ with $J_i=\one\{q_i^+\ne q_i\}$ be the crossing indicator, where $J_i=1$ indicates that a one-cell endpoint jump occurs. Now we can describe the grid-jump mechanism of the QAT.

\begin{proposition}[Realized endpoint loss bound]
\label{prop:realized-crossing-descent}
Suppose Assumptions~\ref{ass:update-direction} and~\ref{ass:smooth} hold, and the stepsize is small enough that $\eta U< \min_i s_i$.  Then
\[
    \mathbf{q}^+-\mathbf{q}
    =
    -\mathbf{s}\odot\sign(\mathbf{u})\odot\mathbf{J},
\]
and
\begin{equation}
\label{eq:realized-crossing-descent-bound}
    f(\mathbf{q}^+)-f(\mathbf{q})
    \le
    -\sum_{\{i:J_i=1\}}s_i\sign(u_i)g_i
    +
    \frac12\sum_{\{i:J_i=1\}}L_i s_i^2 .
\end{equation}
\end{proposition}

Proposition~\ref{prop:realized-crossing-descent} illustrates the finite-grid descent mechanism: once the crossed coordinates are known, the bound certifies endpoint descent whenever the first-order signal exceeds the second-order penalty. The proof is deferred to Appendix~\ref{app:proof-realized-endpoint}.

\section{Residual Homogeneity and Endpoint Motion}
\label{sec:residual-homogeneity}

The bound in Proposition~\ref{prop:realized-crossing-descent} controls the endpoint-loss change for a realized crossing set, but it does not specify how often each coordinate crosses. To quantitatively understand how the crossing mechanism of QAT aggregates to a steady decrease of the endpoint loss and to illustrate when QAT stops making progress, we also need to know how often such crossings occur. Thus, we model the crossing indicators $J_i$ conditionally on the current endpoint and update direction. The resulting conditional laws provide an approximate perspective on QAT by linking the frequency of boundary crossings to expected descent at the quantized endpoint.

\begin{model}[Conditional crossing model]
\label{model:residual-general}
Fix an endpoint $\mathbf{q}=Q(\mathbf{x})$, an update direction $\mathbf{u}$, and a stepsize $\eta$ in a regime where each active coordinate can cross at most one neighboring quantization boundary.  Let $J_i\in\{0,1\}$ be the crossing indicator defined above.  We model $J_i$, conditional on $(\eta,\mathbf{q},\mathbf{u})$, as a Bernoulli random variable with conditional mean
\[
    p_i(\eta,\mathbf{q},\mathbf{u})
    =
    \E[J_i\mid \eta,\mathbf{q},\mathbf{u}]
    =
    \Pr(J_i=1\mid \eta,\mathbf{q},\mathbf{u}) .
\]
The deterministic realized case is included by taking $p_i(\eta,\mathbf{q},\mathbf{u})\in\{0,1\}$.
\end{model}

To further illustrate the inner structure of the endpoint motion, we introduce the residual phase. For $\mathbf{q}=Q(\mathbf{x})$, define the residual $\mathbf{r} = \mathbf{x}-Q(\mathbf{x})$.
In a non-saturated cell, $|r_i|\le \frac{s_i}{2}$, and the normalized residual $\theta_i=r_i/s_i$ lies in $\left[-\frac1 2,\frac1 2\right)$.  This phase is the state variable that decides whether the next shadow update changes the deployed endpoint.  If $x_i$ is near the center of its cell, a small shadow update leaves $q_i$ unchanged; if $x_i$ is within a boundary window of width $\eta |u_i|$, the same update produces a one-cell endpoint jump.
We make this precise in Corollary~\ref{cor:crossing-probability} (see Appendix~\ref{app:crossing-probability}): the crossing event is exactly the
event that $\theta_i$ falls in a window of normalized width $\rho_i = \eta|u_i|/s_i$ at the boundary
the update moves toward, so $p_i$ is the mass the conditional phase law assigns to that window.

In LLM quantization, it is empirically observed that after layerwise or blockwise rescaling, the central mass of pretrained weights is often close to zero-centered normal distributions.  QLoRA \citep{dettmers2023qlora} uses this observation to motivate NormalFloat ($\mathrm{NF4}$), a low-bit datatype designed for normally distributed weights. This observation is also related to mean-field analyses of wide neural-network training, where the empirical measure of the parameter admits a large-width limiting description and cross-particle dependence becomes weak under suitable scaling \citep{sirignano2020meanfield,debortoli2020quantitative}. Corresponding pretrained-weight diagnostics for Qwen-2.5-0.5B are reported in Appendix~\ref{app:pretrained-normal-weight-diagnostics}.

\begin{figure}[t]
\vspace{-1.2em}
\centering
\includegraphics[width=0.9\linewidth]{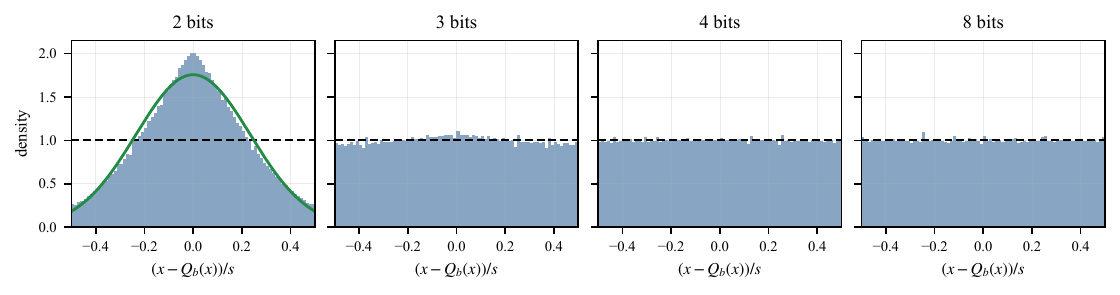}
\caption{Distributions of the normalized residual $\theta_i=(x_i-Q(\mathbf{x})_i)/s_i$ for Qwen-2.5-0.5B at the pretrained point, pooled over non-saturated linear weights.  The solid curve in the 2-bit panel is a fitted zero-centered truncated normal.  The dashed reference is the uniform cell-phase density on $\left[-\frac1 2,\frac1 2\right)$.  The 3-, 4-, and 8-bit cases are close to the uniform residual-phase calculation underlying Model~\ref{model:residual}.}
\label{fig:qwen25-0p5b-residual-homogeneity}
\vspace{-1.2em}
\end{figure}

Based on this observation, we reason about the residual phase induced by folding weights through a finite grid. Consider a scalar pretrained weight $X\sim\mathcal N(0,\tau^2)$ with $\tau>0$, and a symmetric $b$-bit quantizer with grid spacing $s_b>0$. When the weights are measured in units of one quantization step, $X/s_b$ has standard deviation $\lambda_b=\tau/s_b$. Aggregating the normalized residual phase across all non-saturated interior cells gives a folded density
\[
    \pi(\theta)
    \propto
    \sum_{k=-K_b+1}^{K_b-1}
    \exp\!\left(-\frac{(k+\theta)^2}{2\lambda_b^2}\right), \qquad\theta\in\left[-\frac 1 2, \frac 1 2\right).
\]
For this zero-centered folded model, the mass near the upper and lower cell boundaries is the same. We therefore write the one-sided boundary mass as $B(\rho) = \int_{1/2-\rho}^{1/2}\pi(\theta)\dd\theta = \int_{-1/2}^{-1/2+\rho}\pi(\theta)\dd\theta$ for $0\le \rho<1$. For $b=2$, we have $K_2=1$, so the only non-saturated interior cell is the zero cell and the pooled residual phase is center-concentrated. Such a phase places less mass near the cell boundary than the uniform case, so $B(\rho)<\rho$ for small $\rho$. For $b>2$, pooling over many interior cells makes the folded sum nearly flat within one cell, so $B(\rho)\approx\rho$. This is why the 3-, 4-, and 8-bit histograms in Figure~\ref{fig:qwen25-0p5b-residual-homogeneity} are much closer to a uniform cell-phase density. The arguments above start from a normally distributed weight parameter, but the resulting near-uniformity holds for a broad class of weight densities (see Appendix~\ref{app:near-uniform-residual-phases} for more discussion). In Appendix~\ref{app:two-bit-residual-distribution}, we also report the predicted folded densities together with their empirical counterparts.

\begin{wrapfigure}[8]{r}{0.33\linewidth}
\vspace{-1.5em}
\centering
\includegraphics[width=\linewidth]{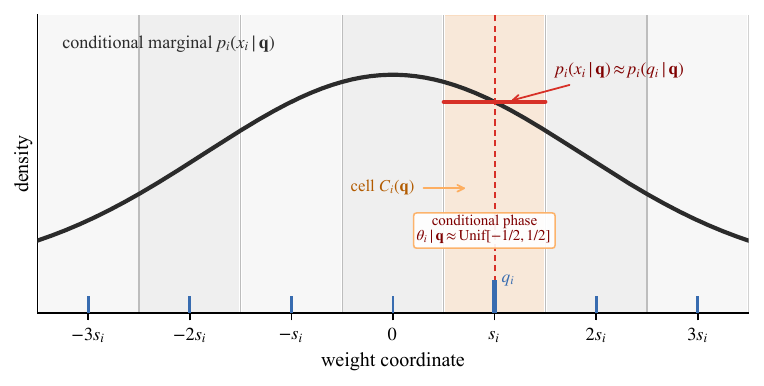}
\captionsetup{font=scriptsize}
\caption{Illustration for residual homogeneity.}
\label{fig:conditional-uniform-cell-illustration}
\vspace{-0.6em}
\end{wrapfigure}

To obtain a tractable description throughout training, we introduce an approximation in which the residual phase remains representative after conditioning on the current endpoint: among non-saturated coordinates, the phase laws do not vary systematically with $\mathbf{q}$. This approximation implies that the update-selected crossing probabilities exhibit the same near-uniform behavior. Figure~\ref{fig:conditional-uniform-cell-illustration} gives the corresponding smooth-density intuition: conditioning on $\mathbf{q}$ restricts each active coordinate $x_i$ to its own quantization cell centered at $q_i$. If the conditional marginal density of $x_i$ varies little across that cell, then the phase $\theta_i=(x_i-q_i)/s_i$ is approximately uniform on $[-\frac12,\frac12)$.  Since the update-selected boundary window has normalized width $\rho_i=\eta |u_i|/s_i$, the uniform phase calculation gives crossing probability $\rho_i$. This gives the simplified conditional model.

\begin{model}[Simplified crossing model]
\label{model:residual}
Suppose Model~\ref{model:residual-general} holds and that, conditional on the current endpoint $\mathbf{q}=Q(\mathbf{x})$, update direction $\mathbf{u}$, and stepsize $\eta$, each non-saturated coordinate has crossing probability
\[
    p_i(\eta,\mathbf{q},\mathbf{u})
    =
    \Pr(J_i=1\mid \eta,\mathbf{q},\mathbf{u})
    =
    \frac{\eta |u_i|}{s_i}.
\]
\end{model}

With the crossing model, we can translate random boundary crossings into an expected one-step bound on the endpoint loss.

\begin{theorem}[Conditional endpoint descent under the crossing model]
\label{thm:active-set-drift}
Suppose Assumptions~\ref{ass:update-direction} and~\ref{ass:smooth} hold, the crossings follow Model~\ref{model:residual-general}, and the stepsize $\eta$ satisfies $\eta U<s_i$ for every coordinate $i$. Given the endpoint $\mathbf{q}=Q(\mathbf{x})$ and update direction $\mathbf{u}=\mathcal{U}(\nabla f(\mathbf{q}))$, write $p_i=p_i(\eta,\mathbf{q},\mathbf{u})$. The endpoint update can be written as
\begin{equation}
\label{eq:stochastic-endpoint-jump}
    \mathbf{q}^+-\mathbf{q}=-\mathbf{s} \odot \sign(\mathbf{u}) \odot \mathbf{J}
\end{equation}
where $J_i\in\{0,1\}$ satisfies $\Pr(J_i=1\mid \eta,\mathbf{q},\mathbf{u})=p_i$. Consequently,
\begin{equation}
\label{eq:general-crossing-endpoint-bound}
    \E[f(\mathbf{q}^+)-f(\mathbf{q})\mid \eta,\mathbf{q},\mathbf{u}]
    \le
    -\sum_i p_i s_i\sign(u_i)g_i
    +
    \frac12\sum_i p_i L_i s_i^2.
\end{equation}
If, in addition, Model~\ref{model:residual} holds and $\mathbf q$ is non-saturated in each coordinate, then
\begin{equation}
\label{eq:coordinate-smooth-endpoint-bound}
    \E[f(\mathbf{q}^+)-f(\mathbf{q})\mid \eta,\mathbf{q},\mathbf{u}]
    \le
    -\eta\ip{\mathbf{g}}{\mathbf{u}}
    +
    \frac{\eta}{2}\sum_iL_i s_i|u_i|.
\end{equation}
\end{theorem}

The proof is deferred to Appendix~\ref{app:proof-conditional-descent}. Theorem~\ref{thm:active-set-drift} is the one-step certificate for quantized-endpoint descent.
Under the uniform crossing model, the first-order term
$-\eta\ip{\mathbf g}{\mathbf u}$ matches that of continuous optimization.  
However, the second-order term is also linear in $\eta$, rather than
$O(\eta^2)$ as in continuous optimization.  This difference arises because a
smaller $\eta$ reduces the probability of a grid jump but not the size of an
accepted jump.  Therefore, making the stepsize arbitrarily small does not by
itself guarantee descent.  The bound therefore certifies descent when the first-order signal is large enough to dominate the
second-order effect of one-cell jumps.

\section{QAR: A Direct Endpoint Algorithmic Framework}
\label{sec:qar}

\subsection{Signal-Imbalance Induced Oscillation in QAT Dynamics}
\label{sec:signal-imbalance-mechanism}
So far, we have used the crossing model to describe how shadow updates induce transitions between deployed endpoints. The simplified crossing model is reasonably well calibrated at higher bit widths, as shown by Figure~\ref{fig:qwen25-0p5b-adam-qat-crossing-fp32-ratio} in Appendix~\ref{app:qwen25-0p5b-crossing-rate}. However, this pooled view can overlook the temporal dependence in QAT dynamics. In particular, adjacent endpoints can have different update signals, causing the shadow weight to spend unequal amounts of time at them. Consider $f(q)=\frac12(q-q^\star)^2$, adjacent endpoints $q_1<q^\star<q_2$, and the shadow update $x_{t+1}=x_t-\eta g(Q(x_t))$, where $g(q)=q-q^\star$. If $\pi_j$ denotes the residence fraction at $q_j$, the small-step two-cell dynamics approximately balance as
\[
    \pi_1 g(q_1)+\pi_2 g(q_2)\approx 0,
    \qquad
    \frac{\pi_2}{\pi_1}
    \approx \frac{|g(q_1)|}{|g(q_2)|}.
\]
Thus, the weaker the signal at an endpoint, the longer the shadow resides there.  In Figure~\ref{fig:qat-signal-imbalance-1d}, the shadow drifts slowly away from $q_1$, but returns rapidly after crossing to $q_2$.  Such signal-imbalanced oscillations have been observed in QAT \citep{nagel2022overcoming}.

\begin{figure}[H]
\vspace{-0.6em}
\centering
\includegraphics[width=0.65\linewidth]{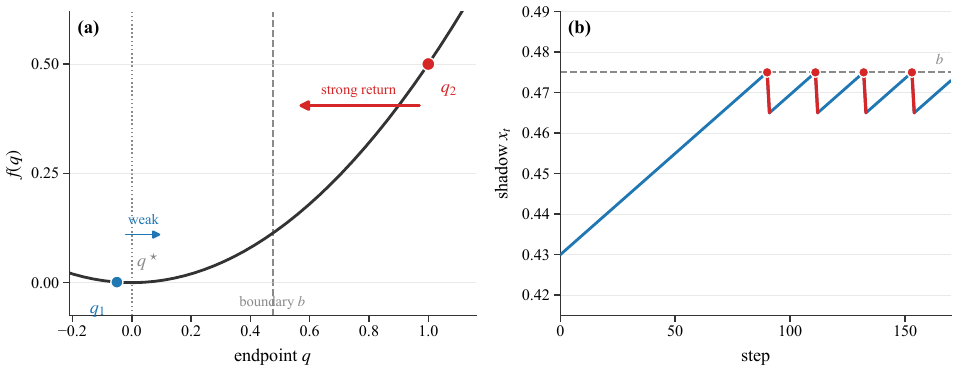}
\caption{Signal-imbalanced QAT dynamics for $f(q)=\frac12q^2$, $q_1=-0.05$, and $q_2=1$.  The weak signal at $q_1$ causes slow drift toward the boundary, whereas the larger signal at $q_2$ quickly pulls the shadow back.}
\label{fig:qat-signal-imbalance-1d}
\vspace{-1.2em}
\end{figure}

Stochastic gradients can further reinforce this asymmetry later in training, when gradient signals become small relative to mini-batch noise. At a weak-signal endpoint such as $q_1$, noise can flip the effective update direction and undo part of the accumulated motion toward the boundary, further lengthening the residence time. Once the shadow crosses to $q_2$, the stronger opposing signal is less easily reversed by noise and quickly drives it back toward $q_1$. This history-dependent cycle biases how long the shadow remains at each endpoint and which crossings persist, even when the pooled crossing rate is well predicted.

This asymmetry can favor the better endpoint, but it is realized indirectly: weak signals move the shadow slowly through a cell, while strong opposite signals can quickly undo a crossing.  This suggests prioritizing endpoint transitions associated with stronger signals.

\subsection{The Proposed Algorithm}
\label{sec:qar-algorithm}
The analysis in Section~\ref{sec:residual-homogeneity} shows how residual
phases convert shadow-weight boundary crossings into probabilistic one-cell
endpoint transitions.  The signal-imbalance mechanism in
Section~\ref{sec:signal-imbalance-mechanism} further suggests assigning more
transition probability to coordinates with stronger signals.  Together,
these two observations motivate a direct realization of quantized-endpoint
dynamics without maintaining a persistent shadow weight.  This leads to
\textbf{QAR} (\textbf{Q}uantization with \textbf{A}mplified
\textbf{R}outing), an algorithmic framework that samples transitions directly
on the quantized codebook. In this framework, a route map selects the update
signal, while an amplifier map converts that signal into coordinatewise jump
intensities. Formally, let
\[
\begin{aligned}
    \mathcal K_i
    =
    \{\underline{k}_i,\underline{k}_i+1,\ldots,\overline{k}_i\}
    \subset\mathbb Z,\quad \mathcal Q=
    \{\mathbf q:\ q_i=s_i k_i,\ k_i\in\mathcal K_i,\ i=1,\ldots,d\},
\end{aligned}
\]
where $\underline{k}_i,\overline{k}_i\in\mathbb Z$ and $\underline{k}_i\le\overline{k}_i$. The symmetric $b$-bit codebook introduced above is the special case $\mathcal K_i=\mathcal K_b$ for every coordinate.

QAR implements these transitions using the integer code $\mathbf k_t$ instead of maintaining a persistent full-precision shadow copy. Under the groupwise quantization used for LLMs, its persistent weight state consists of $\mathbf k_t$ and one shared scale $s_G$ per group, while $\mathbf q_t=\mathbf s\odot\mathbf k_t$ is materialized only for computation. QAR therefore requires less persistent weight state than QAT; detailed memory accounting is provided in Appendix~\ref{app:qwen25-0p5b-memory-diagnostic}.

\begin{algorithm}[t]
\caption{QAR framework: one endpoint step on the finite codebook}
\label{alg:qar}
\begin{algorithmic}[1]
\Require code $\mathbf k_t\in\mathcal K=\mathcal K_1\times\cdots\times\mathcal K_d$, spacings $\mathbf s$, stepsize $\eta$, route map $\mathcal U_t$, amplifier map $\mathcal A$
\State recover endpoint $\mathbf q_t=\mathbf s\odot\mathbf k_t$
\State compute endpoint-gradient estimate $\widehat{\mathbf g}_t$ at $\mathbf q_t$
\State set route $\mathbf u_t=\mathcal U_t(\widehat{\mathbf g}_t)$
\State set direction $\mathbf z_t=\sign(\mathbf u_t)$
\State set feasibility mask $\mathbf M_t=\one_{\mathcal K}(\mathbf k_t-\mathbf z_t)$
\State set probability $\mathbf p_t=\min\{\mathcal A(\eta,\mathbf u_t,\mathbf s,\mathbf M_t),\mathbf 1\}$
\State sample $\mathbf J_t\sim\mathrm{Bernoulli}(\mathbf p_t)$ coordinatewise
\State set $\mathbf k_{t+1}=\mathbf k_t-\mathbf z_t\odot\mathbf J_t$
\Ensure updated code $\mathbf k_{t+1}$
\end{algorithmic}
\end{algorithm}

For a candidate code vector $\mathbf v$, define the coordinatewise feasibility indicator $\one_{\mathcal K}(\mathbf v)_i = \one\{v_i\in\mathcal K_i\}$. Algorithm~\ref{alg:qar} uses an amplifier map $\mathcal A: \mathbb R_+\times\mathbb R^d\times\mathbb R_{++}^d \times\{0,1\}^d \to\mathbb R_+^d$ to convert the gradient magnitude into jump intensities. Because the movement direction is represented separately by $\mathbf z_t=\sign(\mathbf u_t)$, the amplifier returns only nonnegative intensities and assigns zero intensity to infeasible coordinates through $\mathbf M_t$. This masking guarantees that $\mathbf k_{t+1}\in\mathcal K$ whenever $\mathbf k_t\in\mathcal K$, so every QAR iterate is well-defined. Both the minimum and the Bernoulli sampling in Algorithm~\ref{alg:qar} are applied coordinatewise.

A natural baseline is obtained by reproducing the simplified crossing law in Model~\ref{model:residual}. Applying the feasibility mask therefore gives the uniform amplifier
\begin{equation}
\label{eq:uniform-amplifier}
    \mathcal A_{\rm unif}(\eta,\mathbf u,\mathbf s,\mathbf M)
    =
    \eta\mathbf M\odot|\mathbf u|\oslash\mathbf s,
\end{equation}
where $\oslash$ denotes elementwise division. With this choice, Algorithm~\ref{alg:qar} realizes the boundary-projected uniform crossing model: the amplifier assigns zero intensity to outward jumps, while the cap ensures $\mathbf p_t\in[0,1]^d$.  Away from the boundary, it recovers the transition law in Model~\ref{model:residual} whenever $\eta|u_{t,i}|\le s_i$. With $\mathcal A=\mathcal A_{\rm unif}$, the SGD route $\mathbf u_t=\widehat{\mathbf g}_t$, and the one-cell condition $\eta|\widehat g_{t,i}|\le s_i$, QAR recovers the DQT stochastic-rounding update~\citep{zhao2025directquantizedtraining} on a fixed uniform codebook: coordinate $i$ moves one level in direction $-\sign(\widehat g_{t,i})$ with probability $\eta|\widehat g_{t,i}|/s_i$, with matching clipping at codebook boundaries. The detailed comparison with DQT and other direct low-precision methods can be found in Appendix~\ref{app:related-work}.

The QAR framework specifies how to sample endpoint transitions, but it leaves open how the routing intensity should be distributed within each scale-sharing group. The signal-imbalance analysis suggests assigning larger amplification scores to coordinates with stronger feasible gradient magnitudes. For each scale-sharing group $G$ and every $i\in G$, let $s_G=s_i$ denote the common scale within this group and $a_i=M_i|u_i|$ be the feasible gradient magnitude. Define the normalized power score for $\gamma\ge0$ by
\[
    b_{\gamma,i}
    =\begin{cases}
    \displaystyle\frac{a_i^\gamma}
    {|G|^{-1}\sum_{j\in G}a_j^\gamma},
       & \gamma>0\ \text{and }\sum_{j\in G}a_j^\gamma>0,\\[4pt]
    1, & \text{else},
    \end{cases}
    \qquad i\in G,
\]
and define the power amplifier
\begin{equation}
\label{eq:power-amplifier}
    \mathcal A_{\gamma,i}(\eta,\mathbf u,\mathbf s,\mathbf M)
    =\frac{\eta M_i|u_i|}{s_G}\,b_{\gamma,i},
    \qquad i\in G.
\end{equation}
For any $\gamma\ge0$, we call Algorithm~\ref{alg:qar} instantiated with $\mathcal A_\gamma$ by \emph{QAR-$\gamma$}. At $\gamma=0$, it reduces to the uniform-routing update in \eqref{eq:uniform-amplifier} {by treating coordinates within each scale-sharing group equally; for $\gamma>0$, we amplify the crossing probability for coordinates with larger gradient magnitude $a_i=M_i|u_i|$.}

\begin{figure}[H]
    \centering
    \includegraphics[width=0.92\linewidth]{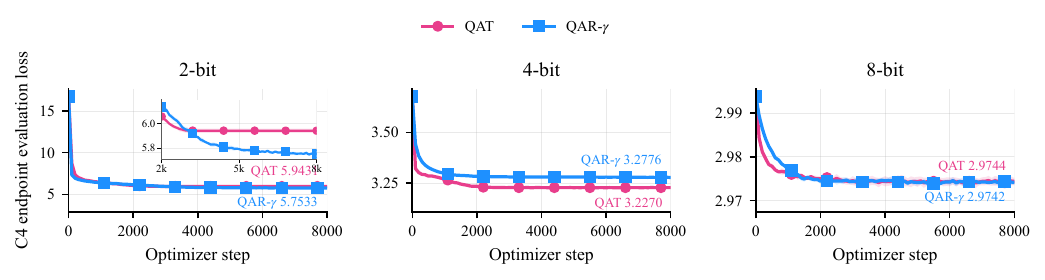}
\caption{Qwen-2.5-0.5B evaluation-loss curves for the best settings of QAT and QAR-$\gamma$ at 2, 4, and 8 bits. Curves show the mean over three seeds, and shading shows $\pm1$ sample standard deviation.}
    \label{fig:qwen25-0p5b-main-qat-qar-best-eval-loss-q}
\end{figure}

\paragraph{Experimental Results.} Figure~\ref{fig:qwen25-0p5b-main-qat-qar-best-eval-loss-q} shows that QAR-$\gamma$ helps most at 2 bits on Qwen-2.5-0.5B. At higher bits, QAR remains competitive with QAT. The selected amplification is strongest at 2 bits, weaker at 4 bits, and absent at 8 bits. This suggests that favoring stronger update signals is useful when grid jumps are coarse, but offers little benefit on finer grids. The Llama-3.2-3B results in Table~\ref{tab:llama3b-qat-qar-8task} show the same trend at a larger scale: QAR-$\gamma$ performs better at 2 bits on both language-modeling and average downstream metrics, while the differences are small at 4 and 8 bits. Table~\ref{tab:memory-measured} further shows that the memory advantage grows as the code width decreases. QAR therefore provides its clearest quality and memory benefits under aggressive quantization.
Detailed experimental settings and ablations are deferred to Appendix~\ref{app:additional-empirical-diagnostics}.

\begin{table}[t]
\centering
\caption{Llama-3.2-3B results after 10,000 training steps. QAR-$\gamma$ uses $\gamma=(8,1,0)$ at $(2,4,8)$ bits.}
\label{tab:llama3b-qat-qar-8task}
\begingroup
\scriptsize
\setlength{\tabcolsep}{2.4pt}
\resizebox{\textwidth}{!}{%
\begin{tabular}{clrrrrrrrrrrr}
\toprule
Bits & Method & C4 PPL $\downarrow$ & Wiki PPL $\downarrow$ & HellaSwag & PIQA & ARC-E & ARC-C & Wino. & BoolQ & OBQA & LAMBADA & Avg. $\uparrow$ \\
\midrule
2 & QAT & 358.89 & 779.54 & 25.32 & 52.77 & 26.77 & 24.83 & 50.36 & 38.17 & 26.00 & 0.00 & 30.53 \\
2 & QAR-$\gamma$ & \textbf{201.90} & \textbf{393.14} & 25.56 & 52.12 & 28.66 & 23.81 & 49.33 & 49.69 & 25.20 & 0.08 & \textbf{31.81} \\
\midrule
4 & QAT & \textbf{14.15} & \textbf{12.13} & 70.22 & 77.15 & 68.86 & 42.32 & 64.40 & 72.72 & 37.60 & 64.22 & \textbf{62.19} \\
4 & QAR-$\gamma$ & \textbf{14.41} & \textbf{12.32} & 69.82 & 76.44 & 68.60 & 40.27 & 67.09 & 71.74 & 38.20 & 64.43 & \textbf{62.07} \\
\midrule
8 & QAT & \textbf{11.09} & \textbf{10.12} & 73.83 & 77.69 & 72.35 & 46.25 & 70.09 & 73.43 & 42.40 & 69.47 & \textbf{65.69} \\
8 & QAR-$\gamma$ & \textbf{11.06} & \textbf{10.12} & 73.72 & 77.97 & 71.89 & 46.42 & 69.85 & 73.36 & 42.80 & 69.77 & \textbf{65.72} \\
\bottomrule
\end{tabular}%
}
\endgroup
\end{table}

\begin{center}
\captionof{table}{Peak CUDA memory cost (MiB) for QAT and QAR-$\gamma$. Parentheses show reductions relative to QAT.}
\label{tab:memory-measured}
\footnotesize
\setlength{\tabcolsep}{4pt}
\renewcommand{\arraystretch}{1.05}
\begin{tabular}{lrrrr}
\toprule
& QAT & \multicolumn{3}{c}{QAR-$\gamma$} \\
\cmidrule(lr){3-5}
& & 8-bit & 4-bit & 2-bit \\
\midrule
Qwen-2.5-0.5B & $3{,}238$ & $2{,}904$ ($10.3\%$) & $2{,}733$ ($15.6\%$) & $\mathbf{2{,}648}$ ($\mathbf{18.2\%}$) \\
Llama-3.2-3B & $17{,}743$ & $15{,}100$ ($14.9\%$) & $13{,}756$ ($22.5\%$) & $\mathbf{13{,}084}$ ($\mathbf{26.3\%}$) \\
\bottomrule
\end{tabular}
\end{center}

\subsection{Finite-Grid Feasible-Gradient Bounds for QAR}
\label{sec:finite-grid-feasible-gradient-bounds}

At a codebook boundary, the full gradient need not vanish because one descent direction can be infeasible. For $q_i=s_ik_i$, define the feasible endpoint gradient by
\begin{equation}
\label{eq:feasible-endpoint-gradient}
    \big(\mathbf g_{\mathcal Q}(\mathbf q)\big)_i
    =
    g_i\,\one\{k_i-\sign(g_i)\in\mathcal K_i\}.
\end{equation}
It retains exactly the components with a feasible one-cell descent move and equals the full gradient away from the boundary.

\begin{theorem}[Feasible-gradient bound for QAR-$\gamma$]
\label{thm:qar-uniform-power-nonconvex}
Suppose Assumptions~\ref{ass:update-direction} and~\ref{ass:smooth} hold. Write $f^*_{\mathcal Q}=\min_{\mathbf q\in\mathcal Q}f(\mathbf q)$.  For each scale-sharing group $G$ with $s_i=s_G$, let $L_G=\max_{i\in G}L_i$. For any $\gamma\ge0$, run QAR-$\gamma$ according to Algorithm~\ref{alg:qar}, and let $b_{t,i}$ denote the power score of coordinate $i$ at step $t$. Assume almost surely that $\eta U b_{t,i}\le s_G$ for every $t<T$ and $i\in G$.

If $\mathbf u_t=\nabla f(\mathbf q_t)$ (GD), then
\begin{equation}
\label{eq:qar-uniform-power-cumulative-gradient-bound}
\begin{aligned}
    \frac1T\sum_{t=0}^{T-1}
    \E\!\left[\|\mathbf g_{\mathcal Q}(\mathbf q_t)\|_2^2\right]\le \frac1T\sum_{t=0}^{T-1}\E\!\left[
      \sum_G\sum_{i\in G}b_{t,i}
      \big(\mathbf g_{\mathcal Q}(\mathbf q_t)\big)_i^2\right]\le \frac{2(f(\mathbf q_0)-f^*_{\mathcal Q})}{\eta T}
    +\frac14\sum_G|G|L_G^2s_G^2.
\end{aligned}
\end{equation}
If $\mathbf u_t=\sign(\nabla f(\mathbf q_t))$ (signGD), then
\begin{equation}
\label{eq:qar-uniform-power-sign-cumulative-gradient-bound}
    \frac1T\sum_{t=0}^{T-1}
    \E\!\left[\|\mathbf g_{\mathcal Q}(\mathbf q_t)\|_1\right]
    \le
    \frac{f(\mathbf q_0)-f^*_{\mathcal Q}}{\eta T}
    +\frac12\sum_G|G|L_Gs_G.
\end{equation}
\end{theorem}

\begin{remark}[Effect of signal imbalance]
\label{rem:power-weighted-certificate}
The first inequality in \eqref{eq:qar-uniform-power-cumulative-gradient-bound} follows from Lemma~\ref{lem:power-score}. It is an equality for $\gamma=0$ and is strict for $\gamma>0$ whenever the feasible-gradient magnitudes are unequal within at least one group. This shows how leveraging signal imbalance can accelerate convergence relative to uniform routing.
\end{remark}

For comparison, let $\mathbf y_{t+1}=\mathbf y_t-\eta\nabla f(\mathbf y_t)$. For an $L$-smooth objective, classical nonconvex GD with fixed $0<\eta\le1/L$ satisfies
\[
\frac1T\sum_{t=0}^{T-1}\|\nabla f(\mathbf y_t)\|_2^2
\le \frac{2(f(\mathbf y_0)-\inf_y f(y))}{\eta T}
=O(T^{-1}).
\]
Equation~\eqref{eq:qar-uniform-power-cumulative-gradient-bound} replaces the full gradient by the feasible endpoint gradient and adds a grid floor, while $\eta U b_{t,i}\le s_G$ ensures that the transition probabilities are not truncated. For gradient update, the bound of QAR decreases like $T^{-1}$ until it reaches the fixed-grid floor $\frac14\sum_G|G|L_G^2s_G^2$. The sign update instead controls the average feasible $\ell_1$ gradient by $O(T^{-1})$ plus its grid floor. Both finite-grid floors are unavoidable under the stated assumptions. The proof and a construction showing that these finite-grid floors are tight appear in Appendix~\ref{app:power-amplifier-theory}. 
The stochastic extension is deferred to Appendix~\ref{app:stochastic-qar}, where we further characterize the tradeoff between exploiting signal imbalance and amplifying stochastic-gradient noise. 

\section{Conclusion and Limitations}
QAT updates a full-precision shadow weight but deploys a quantized endpoint, and the relevant object of analysis is therefore the deployed endpoint trajectory. We show how homogeneous residual model turns endpoint motion into a probabilistic transition on the finite grid and how signal imbalance can create asymmetric oscillations, motivating QAR's update on the quantization code without keeping a full-precision shadow weight. We derive feasible-gradient bounds for QAR with power amplifier and discuss the tradeoff between the gain from signal imbalance and the penalty of strengthening stochastic noise (see Appendix~\ref{app:stochastic-qar}). Experiments support the endpoint view, especially in low-bit LLM post-training, and show that QAR is comparable to or better than QAT at lower peak memory.

\noindent\textbf{Limitations.} This work focuses on weight-only quantization. In this setting, the deployed model lives on a finite weight grid. Our analysis tracks how STE updates move the shadow weights across the cells of this grid. Practical low-bit QAT often quantizes activations as well, where the quantized values depend on the input data and the endpoint can change through both cell crossings and activation-level changes.  Extending the finite-grid endpoint view to joint weight-activation QAT is an important direction for future work.

\newpage
\bibliographystyle{iclr2027_arxiv}
\bibliography{references}

\newpage
\appendix
\raggedbottom
\makeatletter
\renewcommand{\@seccntformat}[1]{%
  \ifnum\pdfstrcmp{#1}{section}=0
    Appendix~\csname the#1\endcsname\quad
  \else
    \csname the#1\endcsname\quad
  \fi
}
\@addtoreset{table}{section}
\makeatother
\setcounter{table}{0}
\renewcommand{\thetable}{\Alph{section}.\arabic{table}}

\section{Related Work}
\label{app:related-work}

\paragraph{Direct low-precision weight updates.} Direct updates of discrete weights predate recent low-bit language-model training.  Discrete State Transition (DST) uses probabilistic projection to keep weights in a configurable multilevel discrete space \citep{li2018dst}. In the one-cell regime, DST's nonlinear projection can be represented within QAR by a tanh amplifier.  Bop instead stores binary weights and an exponential moving average of their gradients, then deterministically flips a weight when the accumulated signal exceeds a threshold and points in the flipping direction \citep{helwegen2019latent}. Direct Quantized Training (DQT) applies stochastic rounding to a tentative optimizer update on a low-bit language-model weight grid \citep{zhao2025directquantizedtraining}. On a fixed uniform codebook, QAR-$0$ with the SGD route and $\eta|u_{t,i}|\le s_i$ recovers its coordinatewise transition law and boundary clipping.  ECO removes master weights through an error-feedback loop in optimizer momentum \citep{nikdan2026eco}.  QAR approaches direct discrete updates from a different starting point: it derives a transition law from QAT's boundary-crossing dynamics and factors that law into a route and an amplifier.  This formulation supports signal-dependent power amplification and a finite-codebook analysis in terms of feasible endpoint gradients.

\paragraph{Transition-rate scheduling.} One way to regulate endpoint motion in QAT is to control how often shadow weights cross quantization boundaries. \citet{lee2025scheduling} adapt the shadow-weight learning rate to track a prescribed fraction of quantized weights that change levels. QAR instead samples coordinatewise endpoint transitions directly.

\paragraph{Update scaling.} Boundary crossings can also be influenced by changing the magnitude or frequency of shadow-weight updates. \citet{lee2021ewgs} use the gradient sign and quantization residual to rescale each STE update, with a Hessian-based rule for adapting the scaling factor. \citet{joo2026maskingupdates} randomly mask blockwise optimizer updates while retaining dense moment updates and use momentum--gradient alignment to control the masking probability. These approaches modify shadow-weight updates, whereas QAR leverages signal imbalance to assign larger transition probabilities to feasible one-cell endpoint updates with stronger gradient signals.

\section{Experimental Results}
\label{app:additional-empirical-diagnostics}

\begin{table}[!htbp]
\caption{Experimental settings for the Qwen-2.5-0.5B and Llama-3.2-3B runs}
\label{tab:llm-experimental-settings}
\label{tab:qwen25-0p5b-qat-qar-setting}
\label{tab:llama32-3b-experimental-setting}
\centering
\scriptsize
\setlength{\tabcolsep}{3pt}
\begin{tabular}{p{0.18\linewidth}p{0.36\linewidth}p{0.36\linewidth}}
\toprule
setting & Qwen-2.5-0.5B & Llama-3.2-3B \\
\midrule
compute resource & 1 NVIDIA A40 GPU per run & 1 NVIDIA A100 GPU per run \\
training budget & $8000$ steps & $10{,}000$ steps \\
sequence length & $512$ & $512$ \\
quantizer & \multicolumn{2}{p{0.72\linewidth}}{fixed symmetric per-output-channel uniform weight quantization, initialized by RTN; language-model head excluded} \\
bit widths & $2,4,8$ & $2,4,8$ \\
scale selection & \multicolumn{2}{p{0.72\linewidth}}{$s_c=\max_j |W_{cj}|/(2^{b-1}-1)$, fixed during training} \\
training scope & \multicolumn{2}{p{0.72\linewidth}}{quantized linear weights only; biases, embeddings, normalizations, and all other parameters frozen} \\
training data & \multicolumn{2}{p{0.72\linewidth}}{tokens from SlimPajama-627B~\citep{cerebras2023slimpajama}} \\
batch size & $1$ & $4$ \\
seeds & $0,1,2$ & $0$ \\
evaluation data & C4 validation~\citep{raffel2020exploring} & C4 and WikiText-2~\citep{merity2017pointer} perplexity; eight zero-shot tasks \\
optimizer & \multicolumn{2}{p{0.72\linewidth}}{AdamW with QAT $(\beta_1,\beta_2)=(0.9,0.95)$ and QAR-$\gamma$ $(\beta_1,\beta_2)=(0.95,0.9)$, $\epsilon=10^{-8}$} \\
weight decay & \multicolumn{2}{p{0.72\linewidth}}{$0$} \\
gradient clipping & \multicolumn{2}{p{0.72\linewidth}}{$0$ (disabled)} \\
learning-rate schedule & \multicolumn{2}{p{0.72\linewidth}}{no warmup; cosine decay to $0.1\eta_0$ over the first half of the training budget, then held fixed} \\
\bottomrule
\end{tabular}
\end{table}

\subsection{Experimental Settings}
\label{app:qwen25-0p5b-exp}

We use Qwen-2.5-0.5B and Llama-3.2-3B to evaluate the finite-grid predictions on post-training QAT runs and compare QAT with QAR-$\gamma$. For each model, both methods start from the same pretrained checkpoint and run at bit widths $b\in\{2,4,8\}$.  We quantize linear-layer weights except the language-model head using fixed symmetric per-output-channel uniform quantization initialized by round-to-nearest (RTN).  For each output channel $c$, we set $s_c=\max_j |W_{cj}|/q_{\max}$, where $W_{cj}$ is the weight in output channel $c$ and input coordinate $j$, and $q_{\max}=2^{b-1}-1$, then round and clip the integer levels to $[-q_{\max},q_{\max}]$.  The scales remain fixed for both methods. We use the same quantizer family at every bit width so that bit width is the only factor that varies across the comparisons. At 2 bits this is coarser than the sub-channel group sizes typically used for 2-bit deployment, and the absolute quality of the 2-bit endpoints is correspondingly low (Table~\ref{tab:llama3b-qat-qar-8task}). The 2-bit rows should therefore be read as a comparison of endpoint optimization on the coarsest grid rather than as a 2-bit deployment result.

We train both models on SlimPajama-627B~\citep{cerebras2023slimpajama} and use C4 validation~\citep{raffel2020exploring} to evaluate the endpoint. For Llama-3.2-3B, we additionally report WikiText-2 perplexity~\citep{merity2017pointer} and eight zero-shot tasks. All runs use sequence length $512$ and gradient accumulation $1$. The Qwen runs use minibatch size $1$ for $8000$ steps on an NVIDIA A40, whereas the Llama runs use minibatch size $4$ for $10{,}000$ steps on an NVIDIA A100.

All runs use AdamW with QAT $(\beta_1,\beta_2)=(0.9,0.95)$ and QAR-$\gamma$ $(\beta_1,\beta_2)=(0.95,0.9)$, $\epsilon=10^{-8}$, zero weight decay, and no gradient clipping. The learning rate decays without warmup to $0.1\eta_0$ over the first half of the training budget and remains fixed thereafter.

\begin{table}[!htbp]
\caption{Final C4 endpoint evaluation loss at each method's selected best learning rate. Every entry is the mean $\pm$ sample standard deviation over three seeds (lower is better).}
\label{tab:qwen25-0p5b-matched-final-loss}
\centering
\small
\setlength{\tabcolsep}{4pt}
\begin{tabular}{cccc}
\toprule
Bits & QAT & QAR-$\gamma$ & $\gamma$ \\
\midrule
2 & $5.9431\pm0.0175$ & $\mathbf{5.7533\pm0.0120}$ & 16 \\
4 & $\mathbf{3.2270\pm0.0022}$ & $3.2776\pm0.0037$ & 2 \\
8 & $2.9744\pm0.0008$ & $2.9742\pm0.0007$ & 0 \\
\bottomrule
\end{tabular}
\end{table}

Table~\ref{tab:qwen25-0p5b-matched-final-loss} reports the final-step values at the selected matched cosine-schedule settings.

\subsection{Separate Endpoint Diagnostics}

Figure~\ref{fig:qwen25-0p5b-app-qat-shadow-endpoint-loss} directly compares QAT's shadow and deployed-endpoint losses at the selected best learning rate for each bit width. At 2 and 4 bits, the C4 evaluation loss decreases, whereas the shadow loss increases substantially.

\begin{figure}[H]
\centering
\includegraphics[width=\linewidth]{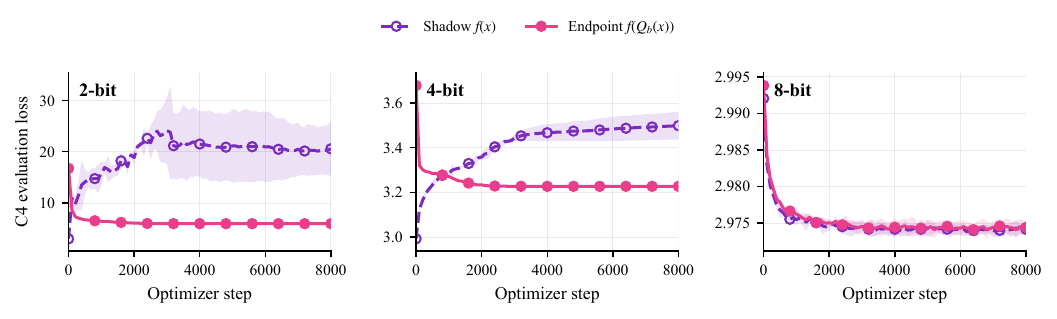}
\caption{Qwen-2.5-0.5B QAT shadow loss $f(x)$ versus deployed-endpoint loss $f(Q_b(x))$. We plot the mean value across 3 seeds, and the shading is $\pm1$ sample standard deviation over three seeds at every bit width.}
\label{fig:qwen25-0p5b-app-qat-shadow-endpoint-loss}
\end{figure}

Figure~\ref{fig:qwen25-0p5b-app-loss-landscape-step-grid} compares the initial and final QAT checkpoints. At initialization, the shadow point has lower loss than the quantized endpoint, especially at 2 bits. After QAT, this ordering is reversed: the endpoint loss falls from $16.75$ to $5.92$ at 2 bits and from $3.68$ to $3.23$ at 4 bits, while the corresponding shadow loss rises from $2.99$ to $26.68$ and $3.47$, respectively. Thus QAT can improve the deployed endpoint even as its shadow moves into a higher-loss region, with the separation becoming most pronounced under aggressive quantization.

\begin{figure}[H]
\centering
\includegraphics[width=\linewidth]{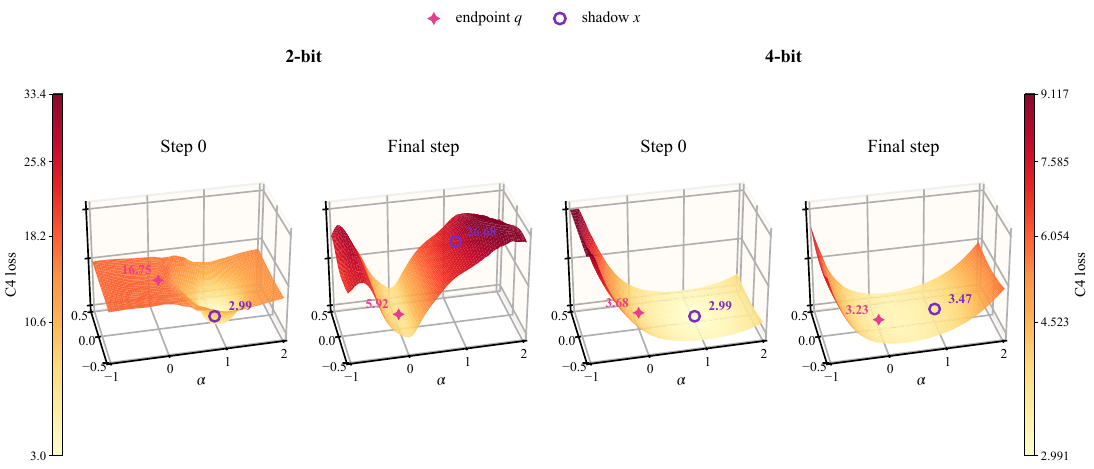}
\caption{Qwen-2.5-0.5B C4 evaluation-loss landscape at step 0 and at final step for seed 0. The surface evaluates $Q_b(\mathbf{x})+\alpha(\mathbf{x}-Q_b(\mathbf{x}))+\beta \norm{\mathbf{x}-Q_b(\mathbf{x})}_2\mathbf v$ for a random unit $\mathbf v\perp(\mathbf{x}-Q_b(\mathbf{x}))$.}
\label{fig:qwen25-0p5b-app-loss-landscape-step-grid}
\end{figure}

\subsection{Gradient Diagnostics}

Figure~\ref{fig:qwen25-0p5b-app-qat-qar-active-l1} tracks the average magnitude of active gradient coordinates, $\|g\|_1/\|\operatorname{sign}(g)\|_1$. At the quantized endpoint, both QAT and QAR-$\gamma$ reduce this quantity, most clearly at 2 bits, showing that both methods remove strong endpoint-gradient signals. In contrast, QAT's shadow-space value can remain much larger at low precision, reinforcing that the shadow trajectory is not a faithful diagnostic of endpoint progress.

\begin{figure}[H]
\centering
\includegraphics[width=\linewidth]{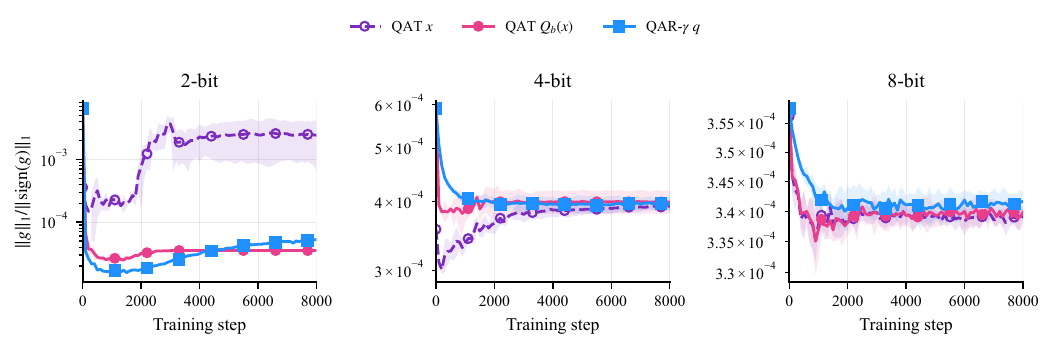}
\caption{Qwen-2.5-0.5B active-gradient ratio at 2, 4, and 8 bits: QAT at $x$, QAT at $Q_b(x)$, and QAR-$\gamma$ at $q$, with $\gamma=(16,2,0)$, respectively. We plot the mean value across 3 seeds, and the shading is $\pm1$ sample standard deviation over three seeds at every bit width.}
\label{fig:qwen25-0p5b-app-qat-qar-active-l1}
\end{figure}

\subsection{Crossing-Rate Diagnostics}
\label{app:qwen25-0p5b-crossing-rate}

For the realized shadow step from $\mathbf{x}_t$ to $\mathbf{x}_{t+1}$, let $\mathcal I_t$ contain the active non-saturated coordinates and define the predicted and observed aggregate crossing rates by
\[
\widehat p_t=\frac{1}{|\mathcal I_t|}\sum_{i\in\mathcal I_t}\min\!\left\{\frac{|(\mathbf{x}_{t+1})_i-(\mathbf{x}_t)_i|}{s_i},1\right\},
\qquad
\widehat c_t=\frac{1}{|\mathcal I_t|}\sum_{i\in\mathcal I_t}\one\!\left\{Q_b(\mathbf{x}_{t+1})_i\ne Q_b(\mathbf{x}_t)_i\right\}.
\]

\begin{figure}[H]
\centering
\includegraphics[width=0.9\linewidth]{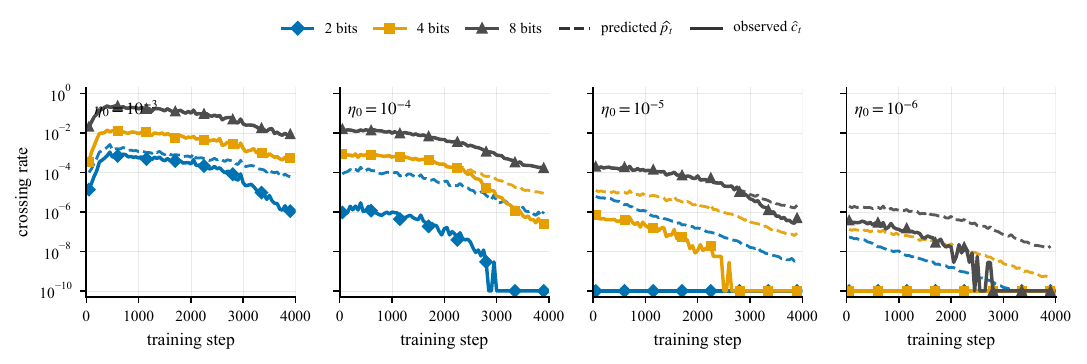}
\caption{QAT crossing rates on Qwen-2.5-0.5B across four learning rates and 2, 4, and 8 bits. Dashed curves show the prediction $\widehat p_t$ and solid curves show the observed rate $\widehat c_t$. Zero observations are displayed at $10^{-10}$.}
\label{fig:qwen25-0p5b-adam-qat-crossing-rate}
\end{figure}

Figures~\ref{fig:qwen25-0p5b-adam-qat-crossing-rate} and~\ref{fig:qwen25-0p5b-selected-qat-qar-crossing-rate} compare predicted or planned crossing rates with realized fixed-grid crossings. Figure~\ref{fig:qwen25-0p5b-adam-qat-crossing-rate} sweeps $\eta_0\in\{10^{-3},10^{-4},10^{-5},10^{-6}\}$ for QAT, while Figure~\ref{fig:qwen25-0p5b-selected-qat-qar-crossing-rate} shows the best QAT and QAR-$\gamma$ settings at 2, 4, and 8 bits. Once the updates become sufficiently small, the realized QAT rate can collapse even though its prediction remains nonzero.

\begin{figure}[H]
\centering
\includegraphics[width=0.86\linewidth]{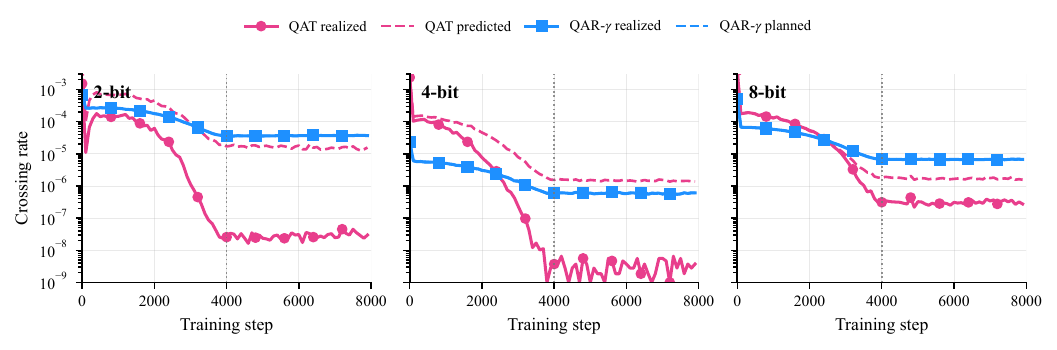}
\caption{Predicted or planned versus realized crossing rates for the best QAT and QAR-$\gamma$ settings. Dashed lines are predicted or planned; solid lines are realized.}
\label{fig:qwen25-0p5b-selected-qat-qar-crossing-rate}
\end{figure}

The late-stage collapse in Figures~\ref{fig:qwen25-0p5b-adam-qat-crossing-rate} and~\ref{fig:qwen25-0p5b-selected-qat-qar-crossing-rate} does not indicate a failure of the residual-phase model; it comes from BF16 rounding. BF16 values are not evenly spaced, and the gap between adjacent values grows with their magnitude. When an update is smaller than half the local gap around $x_i$, round-to-nearest maps $x_i+\Delta x_i$ back to $x_i$, so the shadow weight neither moves nor crosses a quantization boundary. More coordinates enter this regime as the learning rate decays. Figure~\ref{fig:qwen25-0p5b-adam-qat-crossing-fp32-ratio} repeats the diagnostic with FP32 shadow weights to separate this numerical effect from residual-phase calibration.

\begin{figure}[H]
\centering
\includegraphics[width=0.9\linewidth]{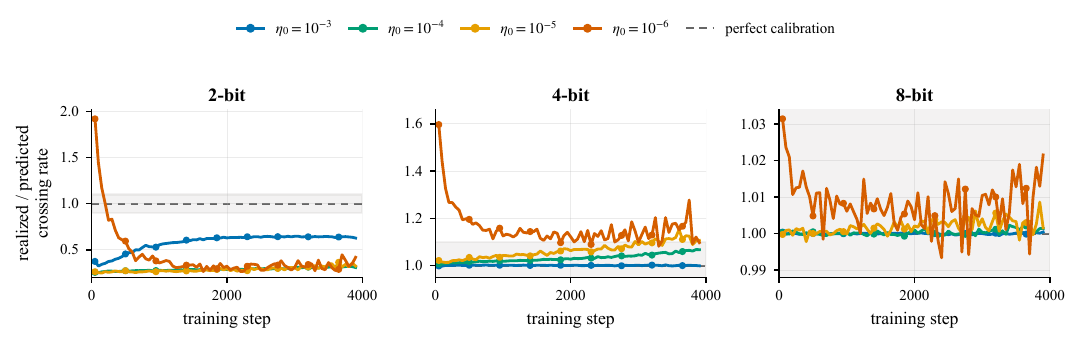}
\caption{FP32 control for Figure~\ref{fig:qwen25-0p5b-adam-qat-crossing-rate}. Curves show the ratio of the realized to the predicted boundary-crossing rate for each initial learning rate; the dashed line marks perfect calibration and the shaded band marks $\pm10\%$.}
\label{fig:qwen25-0p5b-adam-qat-crossing-fp32-ratio}
\end{figure}

With FP32, the late-training collapse disappears. At 4 and 8 bits, the realized-to-predicted ratios stay near one or modestly above it. Their gradual upward drift after the initial transient is consistent with the toy-model analysis of \citet{wenshoj2025oscillations}, in which a residual-driven STE component pushes shadow weights away from quantization levels and toward cell boundaries. This boundary concentration puts more residual mass in the crossing region than the uniform-phase model predicts, causing the realized crossing rate to exceed the prediction. The 2-bit ratios remain below one, consistent with the center-concentrated phase of its single non-saturated interior cell.

At $\eta_0=10^{-6}$, Figure~\ref{fig:qwen25-0p5b-adam-qat-crossing-fp32-ratio} shows an initial spike because the run starts from a BF16 checkpoint. BF16 stores weights on a discrete lattice, placing about $0.1\%$--$0.2\%$ of them exactly on quantization boundaries. Even a tiny first update can move these weights across a boundary, producing more crossings than the continuous phase model predicts. The spike fades once training moves the weights off the BF16 lattice. Thereafter, the realized-to-predicted ratios stay near one or slightly above it at 4 and 8 bits.

QAR does not depend on shadow-weight precision because it samples and applies code transitions directly, without accumulating updates in a shadow weight. It therefore follows the crossing law more faithfully than QAT.

\subsection{Memory Diagnostic}
\label{app:qwen25-0p5b-memory-diagnostic}

\paragraph{Persistent state.} Let $P$ be the total number of model parameters, $N$ the number of quantized coordinates updated by both methods, and $R$ the number of quantization groups. Both methods freeze the remaining $P-N$ parameters. If a shadow value, optimizer moment, and group scale each use $w$, $m$, and $r$ bytes, then
\[
 M_{\rm QAT}=w(P-N)+wN+2mN+rR,
 \qquad
 M_{\rm QAR}=w(P-N)+\frac{b}{8}N+2mN+rR,
\]
where QAR-$\gamma$ stores $b$-bit endpoint codes instead of a shadow tensor.  Its persistent-state saving is therefore $(w-b/8)N$ bytes.

\begin{table}[!htbp]
\caption{Numbers of total parameters $P$, quantized parameters $N$, and quantization groups $R$.}
\label{tab:qwen25-0p5b-memory-diagnostic}
\centering
\small
\setlength{\tabcolsep}{12pt}
\begin{tabular}{lrrr}
\toprule
Model & $P$ & $N$ & $R$ \\
\midrule
Qwen-2.5-0.5B & $494{,}032{,}768$ & $357{,}826{,}560$ & $304{,}128$ \\
Llama-3.2-3B & $3{,}212{,}749{,}824$ & $2{,}818{,}572{,}386$ & $774{,}144$ \\
\bottomrule
\end{tabular}
\end{table}

With BF16 weights and moments and FP32 group scales, the two totals are $2(P-N)+6N+4R$ and $2(P-N)+(4+b/8)N+4R$. Our implementation stores 2- and 4-bit endpoint codes in packed form and updates them with fused packed CUDA kernels. The resulting persistent-state savings at 2, 4, and 8 bits are $25.9\%$, $22.2\%$, and $14.8\%$ on Qwen-2.5-0.5B and $27.9\%$, $23.9\%$, and $15.9\%$ on Llama-3.2-3B, respectively.

\paragraph{Peak memory.} Peak allocation also depends on transient gradient buffers. In our implementation, we adopt a streamed update in which QAR consumes and releases each layer gradient during backward, while its fused sampler avoids model-shaped temporaries. For a fair comparison, we additionally stream the QAT AdamW update in the same way.

\subsection{Learning-Rate Ablation}
\label{app:qwen25-0p5b-lr-ablation}

\begin{figure}[htbp]
    \centering
    \includegraphics[width=0.9\linewidth]{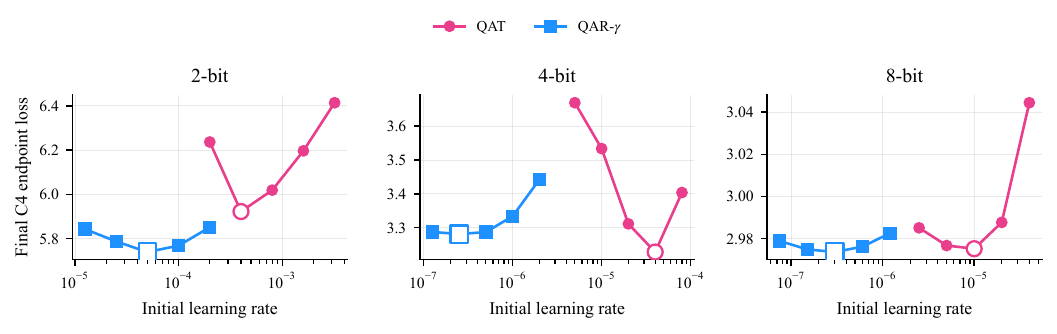}
\caption{Learning-rate sensitivity. Hollow markers denote the lowest final C4 endpoint loss within each displayed method curve.}
    \label{fig:qwen25-0p5b-app-qat-qar-lr-ablation}
\end{figure}

Figure~\ref{fig:qwen25-0p5b-app-qat-qar-lr-ablation} shows that QAR-$\gamma$ operates at a systematically smaller learning-rate scale. It is also less sensitive to the learning rate over each displayed sweep.

\subsection{Power-Exponent Ablation}
\label{app:qwen25-0p5b-gamma-ablation}

\begin{figure}[H]
    \centering
    \includegraphics[width=0.9\linewidth]{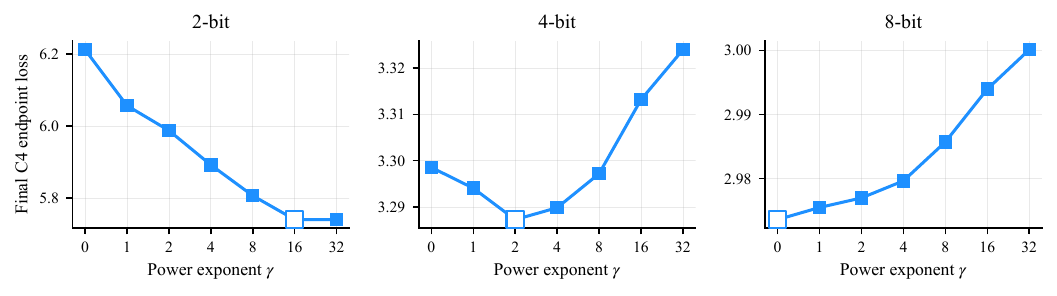}
\caption{QAR-$\gamma$ exponent ablation for $(2,4,8)$ bits. Hollow markers denote the lowest final C4 endpoint loss in each evaluated grid. Strong amplification helps on the 2-bit grid, moderate amplification is best at 4 bits, and uniform routing ($\gamma=0$) is best at 8 bits.}
    \label{fig:qwen25-0p5b-qar-gamma-ablation}
\end{figure}

Figure~\ref{fig:qwen25-0p5b-qar-gamma-ablation} gives the selection sweep over $\gamma\in\{0,1,2,4,8,16,32\}$. Final loss is minimized at $\gamma=16,2,0$ for 2, 4, and 8 bits, respectively. Thus stronger signal concentration is useful on the coarsest grid, but its value disappears as the grid becomes finer. This reflects the tradeoff between the signal imbalance and the stochastic noise. Because the stepsizes used here satisfy the one-cell condition $\eta|u_{t,i}|\le s_i$ at every step, the $\gamma=0$ arm is exactly the DQT update (Section~\ref{sec:qar-algorithm}), so the sweep also compares against a master-weight-free stochastic-rounding baseline.

\section{Residual-Phase Distributions}
\label{app:residual-phase-distributions}

\subsection{A High-Resolution Source of Uniform Phases}
\label{app:near-uniform-residual-phases}

The main text uses residual phases as a conditional model for endpoint crossings.  The following elementary bound gives a sufficient condition under which an unclipped uniform scalar quantizer has nearly uniform residual phases.

\begin{proposition}[Bounded-variation phase error]
\label{prop:bounded-variation-phase-error}
Let $X$ have density $p$ on $\mathbb R$, and let
\[
    Q_s(x)=s\left\lfloor \frac{x}{s}+\frac12\right\rfloor
\]
be nearest-neighbor uniform quantization with step size $s>0$.  Define
\[
    \theta_s=\frac{X-Q_s(X)}{s} \in \left[-\frac1 2,\frac1 2\right).
\]
If $p$ has bounded variation on $\mathbb R$, then $\theta_s$ has density
\[
    \pi_s(\theta)
    =
    s\sum_{k\in\mathbb Z} p\bigl(s(k+\theta)\bigr),
    \qquad \theta\in[-\tfrac1 2,\tfrac1 2),
\]
and
\[
    \sup_{\theta\in[-1/2,1/2)}
    |\pi_s(\theta)-1|
    \le
    s\,\mathrm{TV}(p).
\]
Here $\mathrm{TV}(p)$ denotes the total variation of $p$ on $\mathbb R$, defined as $\sup_{\{x_j\}}\sum_j |p(x_{j+1})-p(x_j)|$, where the supremum is over finite ordered partitions.  If $p$ is absolutely continuous, then $\mathrm{TV}(p)=\int_{\mathbb R}|p'(x)|\,\dd x$. Consequently,
\[
    d_{\mathrm{TV}}\!\left(
    \operatorname{Law}(\theta_s),\mathrm{Unif}[-\tfrac1 2,\tfrac1 2)
    \right)
    \le
    \frac{s\,\mathrm{TV}(p)}{2},
\]
where $d_{\mathrm{TV}}(\mu,\nu)$ means the total variation distance between two probability distributions $\mu$ and $\nu$. Moreover, the boundary-window mass
\[
    B_s(\rho)=\int_{\frac1 2-\rho}^{\frac1 2}\pi_s(\theta)\dd\theta
\]
satisfies
\[
    |B_s(\rho)-\rho|\le \rho\,s\,\mathrm{TV}(p),
    \qquad 0\le\rho\le1 .
\]
For fixed $p$, these bounds vanish linearly as $s\to0$.  Thus $\theta_s$ converges in total variation to $\mathrm{Unif}\!\left[-\frac1 2,\frac1 2\right)$, and $B_s(\rho)\to\rho$ uniformly for $0\le\rho\le 1$.
\end{proposition}

\begin{proof}
The residual-phase map folds each cell $\left[s\bigl(k-\frac1 2\bigr),s\bigl(k+\frac1 2\bigr)\right)$ onto $\left[-\frac1 2, \frac1 2\right)$.  On the cell centered at $sk$, the inverse branch is $x=s(k+\theta)$, so a change of variables contributes $s\,p(s(k+\theta))$ to the density of $\theta_s$.  Summing over $k\in\mathbb Z$ gives the displayed density.  Now fix $\theta$ and write $x_k=s(k+\theta)$.  The intervals $I_k=\left[x_k-\frac s 2,x_k+\frac s 2\right)$ partition $\mathbb R$.  Since $\int p(x)\dd x=1$,
\[
    \pi_s(\theta)-1
    =
    \sum_{k\in\mathbb Z}
    \int_{I_k}\bigl(p(x_k)-p(x)\bigr)\dd x .
\]
For each interval $I_k$,
\[
    |p(x_k)-p(x)|
    \le
    \mathrm{TV}(p;I_k),
    \qquad x\in I_k,
\]
where $\mathrm{TV}(p;I_k)= \sup_{\{y_j\}\subset I_k}\sum_j|p(y_{j+1})-p(y_j)|$ is the variation of $p$ on $I_k$, with the supremum over finite ordered partitions of $I_k$.  Hence
\[
    |\pi_s(\theta)-1|
    \le
    \sum_{k\in\mathbb Z}s\,\mathrm{TV}(p;I_k)
    \le
    s\,\mathrm{TV}(p).
\]
The total-variation and boundary-window bounds follow by integrating this pointwise density bound over intervals of length $1$ and $\rho$, respectively.
\end{proof}

\subsection{Normality of Pretrained Weights}
\label{app:pretrained-normal-weight-diagnostics}

This subsection gives additional evidence for the pretrained normal-weight phenomenon for Qwen-2.5-0.5B. We collect linear weight tensors, exclude embeddings and the LM head, rescale each layer, and draw parameter-weighted samples before overlaying the standard-normal density. Figure~\ref{fig:pretrained-normal-weight-diagnostics} reports the pooled diagnostic. This plot supports the modeling step used in the main text: after local rescaling, the central pretrained weight mass is well approximated by a centered normal law, which induces the folded residual phase model for finite grids.

\subsection{Two-Bit and Folded Residual Phase Distributions}
\label{app:two-bit-residual-distribution}

The main text uses residual phases to explain when shadow motion becomes endpoint motion.  We now spell out a simple folded-normal model for the pooled residual phase distribution. Let $X\sim\mathcal N(0,\tau^2)$ be a scalar pretrained-weight model, where $\tau$ is the weight standard deviation.  For a symmetric $b$-bit quantizer with clipping threshold $c>0$, write
\[
    K_b=2^{b-1}-1,\qquad
    s_b=\frac{c}{K_b}, \qquad
    \lambda_b=\frac{\tau}{s_b}.
\]
The grid levels are $s_b k$ for integers $-K_b\le k\le K_b$, so $c=K_b s_b$ is the largest positive representable value.  Thus $\lambda_b$ measures the pretrained weight scale in grid-step units. The pooled interior residual phase has density
\[
    \pi(\theta)
    =
    \frac{
        \sum_{k=-K_b+1}^{K_b-1}
        \exp\!\left(-\frac{(k+\theta)^2}{2\lambda_b^2}\right)
    }{
        \int_{-1/2}^{1/2}
        \sum_{k=-K_b+1}^{K_b-1}
        \exp\!\left(-\frac{(k+v)^2}{2\lambda_b^2}\right)\dd v
    },
    \qquad \theta \in \left[-\frac1 2, \frac1 2\right) .
\]
At 2 bits, $K_2=1$, so the sum contains only the zero cell and $\pi(\cdot)$ is the density of a truncated normal on $\left[-\frac1 2, \frac1 2\right)$.  For higher bit widths, the histogram pools many interior cells modulo the grid spacing. With fixed $\tau$ and endpoint $c$, $\lambda_b=K_b\lambda_2$, and the folded sum becomes nearly flat within a cell.

\begin{center}
\centering
\begin{minipage}[t]{0.45\linewidth}
\centering
\includegraphics[width=0.9\linewidth]{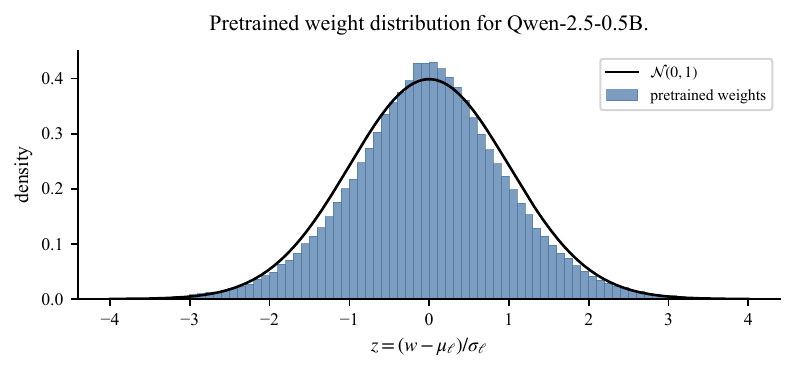}
\captionsetup{hypcap=false} \captionof{figure}{Pooled pretrained weight distribution for Qwen-2.5-0.5B. The plot uses linear weight tensors excluding embeddings and the LM head, standardizes by layer, and draws parameter-weighted samples.} \phantomsection
\label{fig:qwen25-0p5b-normal-weight}
\label{fig:pretrained-normal-weight-diagnostics}
\end{minipage}\hfill
\begin{minipage}[t]{0.52\linewidth}
\centering
\includegraphics[width=\linewidth]{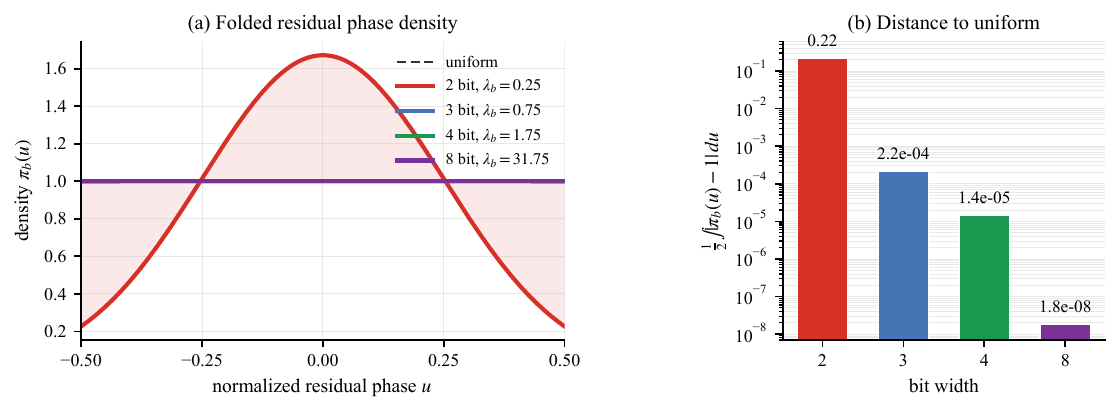}
\captionsetup{hypcap=false} \captionof{figure}{Predicted folded residual phase density $\pi_b(\theta)$ compared with the uniform density on $[-1/2,1/2)$. The left panel uses $\lambda_2=0.25$ and $\lambda_b=K_b\lambda_2$, and the right panel reports the total-variation distance from the uniform model.} \phantomsection
\label{fig:folded-phase-density}
\end{minipage}
\end{center}

Figure~\ref{fig:folded-phase-density} makes the scale separation explicit. With the 2-bit fit $\lambda_2=0.25$, $\pi(\cdot)$ is center-concentrated and has substantial distance from the uniform cell-phase model.  The same absolute weight scale makes the 3-, 4-, and 8-bit folded densities nearly uniform, with total-variation distance dropping by roughly three orders of magnitude already at 3 bits. Figure~\ref{fig:qwen25-0p5b-residual-homogeneity} reports the corresponding empirical Qwen-2.5-0.5B residual histograms at the initial checkpoint.

Figure~\ref{fig:qwen25-0p5b-residual-phase-ks-over-time} shows that QAT changes the residual distribution only modestly after the initial adjustment. The dominant difference is across bit widths: 2-bit remains the least uniform, while 4 and 8 bits stay close to uniform.

\begin{figure}[!htbp]
\centering
\includegraphics[width=0.96\linewidth]{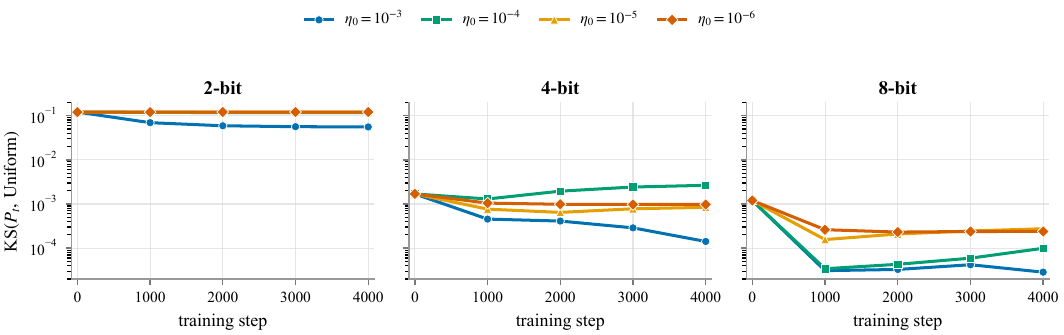}
\caption{KS distance from the uniform residual-phase model over QAT training steps.}
\label{fig:qwen25-0p5b-residual-phase-ks-over-time}
\end{figure}
\FloatBarrier

\section{Proofs for Crossing and Endpoint Motion}
\label{app:proofs-cell-crossing-endpoint-motion}

\subsection{Proof of Lemma~\ref{lem:coordinatewise-crossing-sign}}
\label{app:proof-lemma-alignment}
\begin{proof}[Proof of Lemma~\ref{lem:coordinatewise-crossing-sign}]
Fix coordinate $i$.  If $u_i>0$, then
\[
    x_i^+=x_i-\eta u_i<x_i.
\]
For the fixed coordinatewise uniform quantizer, decreasing $x_i$ cannot increase the endpoint, so
\[
    q_i^+=Q_b(\mathbf{x}^+)_i\le Q_b(\mathbf{x})_i=q_i.
\]
Hence $q_i^+-q_i\le0$, and multiplying by $u_i>0$ gives $u_i(q_i^+-q_i)\le0$.  If $u_i<0$, then $x_i^+>x_i$, so the same fixed grid ordering gives $q_i^+\ge q_i$, and again $u_i(q_i^+-q_i)\le0$.  If $u_i=0$, the product is zero.  If $g_i u_i\ge0$ coordinatewise, then $g_i$ has the same sign as $u_i$ whenever both are nonzero, so $g_i(q_i^+-q_i)\le0$ as well. Summing over coordinates proves the inner-product statement.
\end{proof}

\subsection{Proof of Proposition~\ref{prop:realized-crossing-descent}}
\label{app:proof-realized-endpoint}

\begin{proof}[Proof of Proposition~\ref{prop:realized-crossing-descent}]
Under the stated small-step condition, a crossed active coordinate moves by one neighboring grid level in the direction opposite to $u_i$, while a non-crossed coordinate does not move.  Hence
\[
    q_i^+-q_i=-s_i\sign(u_i)J_i.
\]
Applying Assumption~\ref{ass:smooth} with $\mathbf{x}=\mathbf{q}$ and $\mathbf{y}=\mathbf{q}^+$ gives
\[
    f(\mathbf{q}^+)-f(\mathbf{q})
    \le
    \ip{\mathbf{g}}{\mathbf{q}^+-\mathbf{q}}
    +
    \frac12\sum_iL_i(q_i^+-q_i)^2 .
\]
Substituting the realized jump representation yields
\[
    \ip{\mathbf{g}}{\mathbf{q}^+-\mathbf{q}}
    =
    -\sum_iJ_i s_i\sign(u_i)g_i,
    \qquad
    (q_i^+-q_i)^2=J_i s_i^2,
\]
which proves \eqref{eq:realized-crossing-descent-bound}.  If $g_i u_i\ge0$ and $J_i=1$, then $u_i\ne0$ and $\sign(u_i)g_i=|g_i|$.
\end{proof}

\subsection{Crossing probability of a single coordinate}
\label{app:crossing-probability}

\begin{corollary}[One-coordinate crossing probability]
\label{cor:crossing-probability}
Consider one non-saturated coordinate $i$ with
\[
    r_i\in\left[-\frac{s_i}{2},\frac{s_i}{2}\right),\qquad
    x_i^+=x_i-\eta u_i,\qquad
    \eta |u_i|<s_i.
\]
The endpoint crosses only when $r_i$ lies in the boundary window
\[
    I_i(u_i)
    =
    \begin{cases}
    \left[-\frac{s_i}{2},\,-\frac{s_i}{2}+\eta u_i\right), & u_i>0,\\
    \left[\frac{s_i}{2}+\eta u_i,\,\frac{s_i}{2}\right), & u_i<0,\\
    \varnothing, & u_i=0.
    \end{cases}
\]
On this event,
\[
    q_i^+-q_i=-s_i\,\sign(u_i).
\]
Hence, if $\pi_i(\cdot\mid \mathbf{q})$ is the conditional density of $r_i$, then the crossing probability is
\begin{equation}
\label{eq:general-one-coordinate-crossing-probability}
    p_i
    =
    \int_{I_i(u_i)}\pi_i(v\mid \mathbf{q})\dd v,
    \qquad
    \E[q_i^+-q_i\mid \mathbf{q}]
    =
    -s_i\,\sign(u_i)\,p_i .
\end{equation}
\end{corollary}

\begin{proof}[Proof of Corollary~\ref{cor:crossing-probability}]
Fix coordinate $i$.  If $u_i>0$, the update moves left.  A left crossing occurs precisely when
\[
    r_i-\eta u_i<-\frac{s_i}{2},
\]
which is equivalent to the stated interval.  Since one threshold is crossed, the endpoint moves by $q_i^+-q_i=-s_i$.  The case $u_i<0$ is symmetric: the update moves right, and the crossing condition is
\[
    r_i-\eta u_i\ge\frac{s_i}{2}.
\]
The expectation follows by multiplying the signed jump by its crossing probability.  Integrating the conditional density $\pi_i(\cdot\mid \mathbf{q})$ over the crossing interval gives \eqref{eq:general-one-coordinate-crossing-probability}.
\end{proof}

\subsection{Proof of Theorem~\ref{thm:active-set-drift}}
\label{app:proof-conditional-descent}

\begin{proof}[Proof of Theorem~\ref{thm:active-set-drift}]
Model~\ref{model:residual-general} gives crossing indicators satisfying
\[
    \E[J_i\mid \eta,\mathbf{q},\mathbf{u}]
    =
    p_i(\eta,\mathbf{q},\mathbf{u}).
\]
Under the stated stepsize condition, a crossing in coordinate $i$ has signed jump $-s_i\sign(u_i)$, and no crossing has jump zero.  Therefore
\[
    q_i^+-q_i=-s_i\sign(u_i)J_i.
\]
Collecting these coordinatewise identities proves \eqref{eq:stochastic-endpoint-jump}. Moreover, $(q_i^+-q_i)^2=s_i^2J_i$.
Applying Assumption~\ref{ass:smooth} to the full endpoint jump gives
\[
\begin{aligned}
    \E[f(\mathbf{q}^+)-f(\mathbf{q})\mid \eta,\mathbf{q},\mathbf{u}]
    &\le
    \E\!\left[
        \ip{\mathbf{g}}{\mathbf{q}^+-\mathbf{q}}
        +
        \frac12\sum_iL_i(q_i^+-q_i)^2
        \,\middle|\, \eta,\mathbf{q},\mathbf{u}
    \right]\\
    &=
    -\sum_i p_i(\eta,\mathbf{q},\mathbf{u})s_i\sign(u_i)g_i
    +
    \frac12\sum_i p_i(\eta,\mathbf{q},\mathbf{u})L_i s_i^2.
\end{aligned}
\]
This proves \eqref{eq:general-crossing-endpoint-bound}.  Under Model~\ref{model:residual}, $p_i(\eta,\mathbf{q},\mathbf{u})=\eta|u_i|/s_i$, so the last display becomes \eqref{eq:coordinate-smooth-endpoint-bound}.
\end{proof}

\section{Power-Amplifier Guarantees}
\label{app:power-amplifier-theory}

\subsection{Proof of Theorem~\ref{thm:qar-uniform-power-nonconvex}}

For each scale-sharing group $G$, at an endpoint $q_i=s_Gk_i$, define
\[
    M_i=\one\{k_i-\sign(u_i)\in\mathcal K_i\},
    \qquad a_i=M_i|u_i|,\qquad i\in G.
\]
\begin{equation}
\label{eq:power-score}
    b_{\gamma,i}
    =\begin{cases}
    \displaystyle\frac{a_i^\gamma}
    {|G|^{-1}\sum_{j\in G}a_j^\gamma},
       & \gamma>0\ \text{and }\sum_{j\in G}a_j^\gamma>0,\\[4pt]
    1, & \text{else},
    \end{cases}
    \qquad i\in G.
\end{equation}
Thus the uncapped transition probability is
\begin{equation}
\label{eq:power-probability}
    p_i=\frac{\eta a_i b_{\gamma,i}}{s_G}.
\end{equation}

\begin{lemma}[Power-score properties]
\label{lem:power-score}
For every group $G$ and $\gamma\ge0$,
\begin{equation}
\label{eq:power-score-properties}
    \sum_{i\in G}b_{\gamma,i}=|G|,
    \qquad
    \sum_{i\in G}b_{\gamma,i}a_i
    \ge \sum_{i\in G}a_i,
    \qquad
    \sum_{i\in G}b_{\gamma,i}a_i^2
    \ge \sum_{i\in G}a_i^2.
\end{equation}
For $\gamma>0$, both inequalities are strict exactly when the signals in the group are not all equal.
\end{lemma}

\begin{proof}
The claim is immediate for $\gamma=0$ and for an all-zero group.  Otherwise, the normalization in \eqref{eq:power-score} gives $\sum_i b_{\gamma,i}=|G|$.  For $r\in\{1,2\}$, the maps $z\mapsto z^\gamma$ and $z\mapsto z^r$ are nondecreasing on $\R_+$, so the finite-sum Chebyshev inequality gives
\[
    \frac1{|G|}\sum_{i\in G}a_i^{\gamma+r}
    \ge
    \left(\frac1{|G|}\sum_{i\in G}a_i^\gamma\right)
    \left(\frac1{|G|}\sum_{i\in G}a_i^r\right).
\]
Dividing by the first parenthesized factor proves both inequalities in \eqref{eq:power-score-properties}. Strictness follows when two signals differ.
\end{proof}

For each group, let $L_G=\max_{i\in G}L_i$.  Assumption~\ref{ass:smooth} implies the block majorizer
\begin{equation}
\label{eq:power-block-majorizer}
    f(\mathbf y)
    \le f(\mathbf q)+\langle\nabla f(\mathbf q),\mathbf y-\mathbf q\rangle
    +\frac12\sum_GL_G
       \|\mathbf y_G-\mathbf q_G\|_2^2.
\end{equation}

\begin{proposition}[One-step power-amplifier bound]
\label{prop:power-one-step}
Suppose the route is the exact gradient and the probabilities are uncapped. For $\mathcal A_\gamma$ with any $\gamma\ge0$, let $b_i=b_{\gamma,i}$. Then
\begin{equation}
\label{eq:power-one-step}
\begin{aligned}
    \E[f(\mathbf q^+)-f(\mathbf q)\mid\mathbf q]
    &\le-\eta\sum_G\sum_{i\in G}b_i a_i^2
      +\frac\eta2\sum_G L_Gs_G\sum_{i\in G}b_i a_i\\
    &\le-\frac\eta2\sum_G\sum_{i\in G}b_i a_i^2
      +\frac\eta8\sum_G|G|L_G^2s_G^2\\
    &\le-\frac\eta2\|\mathbf g_{\mathcal Q}(\mathbf q)\|_2^2
      +\frac\eta8\sum_G|G|L_G^2s_G^2.
\end{aligned}
\end{equation}
\end{proposition}

\begin{proof}
For the exact-gradient route, $a_i=|(\mathbf g_{\mathcal Q}(\mathbf q))_i|$.  An accepted jump in coordinate $i\in G$ contributes $-s_Ga_i$ to the linear term of \eqref{eq:power-block-majorizer} and $s_G^2$ to its squared displacement.  Substituting $p_i=\eta a_ib_i/s_G$ gives the first line.  For $c_G=L_Gs_G$, the scalar inequality
\[
    -a_i^2+\frac{c_G}{2}a_i
    \le -\frac12a_i^2+\frac18c_G^2
\]
and $\sum_{i\in G}b_i=|G|$ give the second line.  The final line follows from Lemma~\ref{lem:power-score}, it is an equality at $\gamma=0$.
\end{proof}

\begin{proposition}[One-step signGD bound]
\label{prop:power-sign-one-step}
Suppose $\mathbf u=\sign(\nabla f(\mathbf q))$ and the probabilities are uncapped. For $\mathcal A_\gamma$ with any $\gamma\ge0$, let $b_i=b_{\gamma,i}$. Then
\begin{equation}
\label{eq:power-sign-one-step}
\begin{aligned}
    \E[f(\mathbf q^+)-f(\mathbf q)\mid\mathbf q]
    &\le-\eta\sum_G\sum_{i\in G}
       b_i\big|\big(\mathbf g_{\mathcal Q}(\mathbf q)\big)_i\big|
      +\frac\eta2\sum_GL_Gs_G\sum_{i\in G}b_i a_i\\
    &\le-\eta\|\mathbf g_{\mathcal Q}(\mathbf q)\|_1
      +\frac\eta2\sum_G|G|L_Gs_G.
\end{aligned}
\end{equation}
\end{proposition}

\begin{proof}
Here $a_i=M_i|u_i|\in\{0,1\}$, and $a_i=1$ exactly when the corresponding component of $\mathbf g_{\mathcal Q}(\mathbf q)$ is nonzero. Substituting $p_i=\eta a_ib_i/s_G$ into \eqref{eq:power-block-majorizer} gives the first line. For $\gamma=0$, active coordinates have $b_i=1$ and $\sum_{i\in G}b_ia_i\le|G|$. For $\gamma>0$, if $m_G=\sum_{i\in G}a_i>0$, every active coordinate has $b_i=|G|/m_G\ge1$ and $\sum_i b_ia_i=|G|$; when $m_G=0$, the group does not move. Thus the weighted first-order signal is at least the unweighted feasible-gradient $\ell_1$ norm, while the curvature mass is at most $|G|$ in every group.
\end{proof}

\begin{proof}[Proof of Theorem~\ref{thm:qar-uniform-power-nonconvex}]
Since $a_{t,i}=M_{t,i}|u_{t,i}|\le U$, the condition
$\eta U b_{t,i}\le s_G$ implies
$\eta a_{t,i}b_{t,i}/s_G\le1$.  Hence the coordinatewise minimum in
Algorithm~\ref{alg:qar} is inactive.  Therefore
Proposition~\ref{prop:power-one-step} applies to the GD route and
Proposition~\ref{prop:power-sign-one-step} applies to the signGD route at
every step.
For GD, keep the weighted second line of \eqref{eq:power-one-step}, take total expectation, and sum over $t=0,\ldots,T-1$. Since $f(\mathbf q_T)\ge f^*_{\mathcal Q}$, this gives
\[
\frac1T\sum_{t=0}^{T-1}\E\!\left[
    \sum_i b_{t,i}\big(\mathbf g_{\mathcal Q}(\mathbf q_t)\big)_i^2
\right]
\le
\frac{2(f(\mathbf q_0)-f^*_{\mathcal Q})}{\eta T}
+\frac14\sum_G|G|L_G^2s_G^2.
\]
This proves the second inequality in \eqref{eq:qar-uniform-power-cumulative-gradient-bound}; the first follows by applying Lemma~\ref{lem:power-score} within each group before taking expectations. For signGD, the same summation applied to \eqref{eq:power-sign-one-step} proves \eqref{eq:qar-uniform-power-sign-cumulative-gradient-bound}.
\end{proof}

\subsection{Tightness of the finite-grid floors}
For arbitrary dimension $d$ and positive $L_i,s_i$, consider
\[
    \mathcal Q=\prod_{i=1}^d\{-s_i,0,s_i\},
    \qquad
    f(\mathbf x)=\frac12\sum_{i=1}^d
    L_i\left(x_i-\frac{s_i}{2}\right)^2.
\]
Its Hessian is $\operatorname{diag}(L_1,\ldots,L_d)$, so the coordinate-smoothness constants are exactly $L_i$.  Every $\mathbf q\in \prod_i\{0,s_i\}$ is a global minimizer over $\mathcal Q$.  Moreover, the descent-sign move toggles coordinate $i$ between $0$ and $s_i$, and hence
\[
    \mathbf g_{\mathcal Q}(\mathbf q)=\nabla f(\mathbf q),\qquad
    \|\mathbf g_{\mathcal Q}(\mathbf q)\|_2^2
    =\frac14\|\mathbf L\odot\mathbf s\|_2^2,\qquad
    \|\mathbf g_{\mathcal Q}(\mathbf q)\|_1
    =\frac12\|\mathbf L\odot\mathbf s\|_1.
\]
With exact gradients and $0<\eta\le 2/\max_iL_i$, the QAR-$0$ SGD route moves coordinate $i$ with probability $\eta L_i/2$, so a trajectory initialized in $\prod_i\{0,s_i\}$ remains there.  The initial finite-grid optimality gap and noise terms are zero, while the averaged feasible-gradient measure equals the SGD floor for every $T$.  The same trajectory under the QAR-$0$ signSGD route, with $0<\eta\le\min_i s_i$, attains its $\ell_1$ floor for every $T$.

The exact construction has tied neighboring minimizers, but the scaling does not rely on exact ties.  Replacing $s_i/2$ above by $s_i/2-\delta_i$, where $0<\delta_i<s_i/2$, makes $\mathbf 0$ the unique finite-grid minimizer.  On $\prod_i\{0,s_i\}$, for a sufficiently small positive stepsize, the exact-gradient SGD route has coordinate transition probabilities proportional to $L_i(s_i/2-\delta_i)$ and $L_i(s_i/2+\delta_i)$, its stationary feasible-gradient energy is
\[
    \E_{\pi}\|\mathbf g_{\mathcal Q}(\mathbf q)\|_2^2
    =\sum_{i=1}^d L_i^2\left(\frac{s_i^2}{4}-\delta_i^2\right),
\]
which approaches $\|\mathbf L\odot\mathbf s\|_2^2/4$ as $\boldsymbol\delta\to\mathbf0$.  In particular, in the tied construction with $L_i=L$ and $s_i=s$, the exact floor is $dL^2s^2/4$, even though $\nabla^2f=LI_d$ and the objective is globally $L$-smooth.  Thus the linear dimension dependence is unavoidable for the unnormalized feasible-gradient measure under the stated assumptions.

\section{Stochastic QAR Guarantees}
\label{app:stochastic-qar}

To analyze QAR under stochastic gradient noise, we follow \citet{bernstein2018signsgd} and adopt the standard assumption that the mini-batch gradient $\widehat{\mathbf{g}}(\mathbf{q})$ is unbiased with bounded coordinatewise variance.

\begin{assumption}[Stochastic gradient noise]
\label{ass:stochastic-noise}
At each endpoint $\mathbf q$, the mini-batch gradient is unbiased and has coordinate variance
\[
    \E[(\widehat g_i(\mathbf q)-g_i(\mathbf q))^2\mid\mathbf q]
    \le\sigma_i^2/n.
\]
Given $(\mathbf q_t,\widehat{\mathbf g}_t)$, Algorithm~\ref{alg:qar} samples the jump variables coordinatewise with probabilities $\mathbf p_t$, using fresh routing randomness.
\end{assumption}

\subsection{Baseline stochastic QAR bounds}

\begin{theorem}[Stochastic QAR finite-grid bounds]
\label{thm:stochastic-qar}
Suppose Assumptions~\ref{ass:smooth} and~\ref{ass:stochastic-noise} hold, and write $f^*_{\mathcal Q}=\min_{\mathbf q\in\mathcal Q}f(\mathbf q)$. Fix $\eta>0$ and an integer $T\ge1$. For each scale-sharing group $G$, let $s_i=s_G$ for $i\in G$, $L_G=\max_{i\in G}L_i$, and $\boldsymbol\sigma_G=(\sigma_i)_{i\in G}$. For the SGD route $\mathbf u_t=\widehat{\mathbf g}(\mathbf q_t)$, use $\mathcal A_\gamma$ with $\gamma=0$ and suppose $\eta|\widehat g_{t,i}|\le s_i$ almost surely. Then
\begin{equation}
\label{eq:qat-sgd-cumulative-gradient-bound}
\frac1T\sum_{t=0}^{T-1}\E\|\mathbf g_{\mathcal Q}(\mathbf q_t)\|_2^2
\le
\frac{2(f(\mathbf q_0)-f^*_{\mathcal Q})}{\eta T}
+\frac14\|\mathbf L\odot\mathbf s\|_2^2
+\frac{\langle\mathbf L\odot\mathbf s,\boldsymbol\sigma\rangle}{\sqrt n}
+\frac{2\|\boldsymbol\sigma\|_2^2}{n}.
\end{equation}
For the signSGD route $\mathbf u_t=\sign(\widehat{\mathbf g}(\mathbf q_t))$, use $\mathcal A_\gamma$ for any $\gamma\ge0$, with $\eta\le\min_Gs_G$ when $\gamma=0$ and $\eta\le\min_G(s_G/|G|)$ when $\gamma>0$. Then
\begin{equation}
\label{eq:qat-signsgd-cumulative-gradient-bound}
\frac1T\sum_{t=0}^{T-1}\E\|\mathbf g_{\mathcal Q}(\mathbf q_t)\|_1
\le
\frac{f(\mathbf q_0)-f^*_{\mathcal Q}}{\eta T}
+\begin{cases}
\displaystyle
\frac12\|\mathbf L\odot\mathbf s\|_1
+\frac{2\|\boldsymbol\sigma\|_1}{\sqrt n},
& \gamma=0,\\[6pt]
\displaystyle
\frac12\sum_G|G|L_Gs_G
+\frac{2}{\sqrt n}\sum_G|G|\|\boldsymbol\sigma_G\|_1,
& \gamma>0,
\end{cases}
\end{equation}
\end{theorem}

\begin{proof}
Fix $\mathbf q_t$, write $\mathbf g_t=\nabla f(\mathbf q_t)$, $\widehat{\mathbf g}_t=\mathbf g_t+\boldsymbol{\xi}_t$, and $\mathbf h_t=\mathbf g_{\mathcal Q}(\mathbf q_t)$.  Also write $\mathbf z_t=\sign(\widehat{\mathbf g}_t)$.  The coordinate mask $M_{t,i}$ keeps only feasible one-cell moves, so the endpoint always remains in $\mathcal Q$.

\vspace{0.5em}
\noindent\textbf{SGD route.} For $\mathbf u_t=\widehat{\mathbf g}_t$, conditioning on $(\mathbf q_t,\widehat{\mathbf g}_t)$ and applying Assumption~\ref{ass:smooth} gives
\[
\begin{aligned}
    \E_{\mathbf J}[f(\mathbf q_{t+1})-f(\mathbf q_t)\mid \mathbf q_t,\widehat{\mathbf g}_t]
    &\le
    -\eta\sum_i g_{t,i}\widehat g_{t,i}M_{t,i}
    +
    \frac{\eta}{2}\sum_iL_i s_i|\widehat g_{t,i}|M_{t,i}.
\end{aligned}
\]
We first lower bound the first-order term.  For a coordinate with $h_{t,i}\ne0$, the true descent sign is feasible.  If the coordinate is interior, $M_{t,i}=1$ whenever $\widehat g_{t,i}\ne0$, and hence $g_{t,i}\widehat g_{t,i}M_{t,i}=g_{t,i}\widehat g_{t,i}$.  If the coordinate is at a boundary, the mask removes only stochastic signs for which $g_{t,i}\widehat g_{t,i}\le0$.  Therefore, in both cases,
\[
    \E[g_{t,i}\widehat g_{t,i}M_{t,i}\mid \mathbf q_t]
    \ge
    g_{t,i}^2
    =
    h_{t,i}^2.
\]
For a coordinate with $h_{t,i}=0$, either $g_{t,i}=0$, in which case the term is zero, or the true descent sign is infeasible at a boundary.  In the latter case $g_{t,i}\widehat g_{t,i}M_{t,i}$ is nonzero only after a stochastic sign flip.  On that event, $|g_{t,i}|\le|\xi_{t,i}|$ and $|\widehat g_{t,i}|\le|\xi_{t,i}|$, so
\[
    g_{t,i}\widehat g_{t,i}M_{t,i}
    \ge
    -\xi_{t,i}^2 .
\]
Using the variance bound, for every coordinate,
\[
    \E[g_{t,i}\widehat g_{t,i}M_{t,i}\mid \mathbf q_t]
    \ge
    h_{t,i}^2-\frac{\sigma_i^2}{n}.
\]
The curvature term satisfies
\[
    |\widehat g_{t,i}|M_{t,i}
    \le
    |h_{t,i}|+|\xi_{t,i}|.
\]
Indeed, if $h_{t,i}\ne0$, this follows from $|\widehat g_{t,i}|\le |g_{t,i}|+|\xi_{t,i}|$. If $h_{t,i}=0$ and $g_{t,i}\ne0$, the product $|\widehat g_{t,i}|M_{t,i}$ can be nonzero only on a sign flip, where $|\widehat g_{t,i}|\le|\xi_{t,i}|$.  Hence
\[
    \E[|\widehat g_{t,i}|M_{t,i}\mid \mathbf q_t]
    \le
    |h_{t,i}|+\frac{\sigma_i}{\sqrt n}.
\]
Taking expectation over $\widehat{\mathbf g}_t$ in the one-step bound gives
\[
\begin{aligned}
    \E[f(\mathbf q_{t+1})-f(\mathbf q_t)\mid \mathbf q_t]
    &\le
    -\eta\|\mathbf h_t\|_2^2
    +
    \frac{\eta\|\boldsymbol{\sigma}\|_2^2}{n}
    +
    \frac{\eta}{2}\big\langle\mathbf L\odot\mathbf s,|\mathbf h_t|\big\rangle
    +
    \frac{\eta}{2\sqrt n}
    \big\langle\mathbf L\odot\mathbf s,\boldsymbol{\sigma}\big\rangle .
\end{aligned}
\]
By Young's inequality,
\[
    \frac12\big\langle\mathbf L\odot\mathbf s,|\mathbf h_t|\big\rangle
    \le
    \frac12\|\mathbf h_t\|_2^2
    +
    \frac18\|\mathbf L\odot\mathbf s\|_2^2 .
\]
Thus
\[
    \E[f(\mathbf q_{t+1})-f(\mathbf q_t)]
    \le
    -\frac{\eta}{2}\E\|\mathbf h_t\|_2^2
    +
    \frac{\eta}{8}\|\mathbf L\odot\mathbf s\|_2^2
    +
    \frac{\eta}{2\sqrt n}
    \big\langle\mathbf L\odot\mathbf s,\boldsymbol{\sigma}\big\rangle
    +
    \frac{\eta\|\boldsymbol{\sigma}\|_2^2}{n}.
\]
Summing over $t=0,\ldots,T-1$, using $f(\mathbf q_T)\ge f^*_{\mathcal Q}$, and dividing by $\eta T/2$ proves \eqref{eq:qat-sgd-cumulative-gradient-bound}.

\vspace{0.5em}
\noindent\textbf{signSGD route.} Set $a_{t,i}=M_{t,i}|z_{t,i}|\in\{0,1\}$. For $\gamma=0$, $b_{t,i}=1$. For $\gamma>0$, if $m_{t,G}=\sum_{i\in G}a_{t,i}>0$, then the binary scores give $b_{t,i}=|G|/m_{t,G}$ on active coordinates; if $m_{t,G}=0$, the group does not move. Hence every active score lies in $[1,w_G]$, where $w_G=1$ for $\gamma=0$ and $w_G=|G|$ for $\gamma>0$, and
\[
    \sum_{i\in G}b_{t,i}a_{t,i}\le |G|.
\]
The stated stepsize conditions therefore keep all probabilities uncapped. Conditioning on $(\mathbf q_t,\widehat{\mathbf g}_t)$ and applying Assumption~\ref{ass:smooth} gives
\[
\begin{aligned}
    \E_{\mathbf J}[f(\mathbf q_{t+1})-f(\mathbf q_t)
        \mid \mathbf q_t,\widehat{\mathbf g}_t]
    \le
    -\eta\sum_i b_{t,i}g_{t,i}z_{t,i}M_{t,i}
    +\frac{\eta}{2}\sum_Gs_G
        \sum_{i\in G}L_i b_{t,i}a_{t,i}.
\end{aligned}
\]
For every $i\in G$, a case split over a correct sign, an incorrect sign, and $z_{t,i}=0$ gives
\[
    b_{t,i}g_{t,i}z_{t,i}M_{t,i}
    \ge |h_{t,i}|-2w_G|g_{t,i}|
    \one\{z_{t,i}\ne\sign(g_{t,i})\}.
\]
This also covers a boundary where the true descent direction is infeasible: then $h_{t,i}=0$, and a harmful inward move requires a sign error.  For $g_{t,i}\ne0$,
\[
    \Pr(z_{t,i}\ne\sign(g_{t,i})\mid \mathbf q_t)
    \le
    \Pr(|\widehat g_{t,i}-g_{t,i}|\ge |g_{t,i}|\mid \mathbf q_t)
    \le
    \frac{\sigma_i}{\sqrt n\,|g_{t,i}|},
\]
where the last step uses Markov's inequality on $|\widehat g_{t,i}-g_{t,i}|$ and Cauchy--Schwarz with the coordinate variance bound. Coordinates with $g_{t,i}=0$ contribute zero, so
\[
    \E\left[
        \sum_i b_{t,i}g_{t,i}z_{t,i}M_{t,i}
        \,\middle|\,\mathbf q_t
    \right]
    \ge
    \|\mathbf g_{\mathcal Q}(\mathbf q_t)\|_1
    -
    \frac{2}{\sqrt n}\sum_Gw_G\|\boldsymbol\sigma_G\|_1.
\]
For the curvature term, $\gamma=0$ retains the sharper bound $\sum_iL_is_i$, while $\gamma>0$ is bounded by $\sum_G|G|L_Gs_G$. Taking expectation over the stochastic gradient therefore gives
\[
    \E[f(\mathbf q_{t+1})-f(\mathbf q_t)\mid \mathbf q_t]
    \le
    -\eta\|\mathbf h_t\|_1
    +\begin{cases}
    \displaystyle
    \frac{\eta}{2}\|\mathbf L\odot\mathbf s\|_1
    +\frac{2\eta\|\boldsymbol\sigma\|_1}{\sqrt n},
    & \gamma=0,\\[6pt]
    \displaystyle
    \frac{\eta}{2}\sum_G|G|L_Gs_G
    +\frac{2\eta}{\sqrt n}\sum_G|G|\|\boldsymbol\sigma_G\|_1,
    & \gamma>0.
    \end{cases}
\]
Summing over $t=0,\ldots,T-1$, using $f(\mathbf q_T)\ge f^*_{\mathcal Q}$, and dividing by $\eta T$ proves \eqref{eq:qat-signsgd-cumulative-gradient-bound}.
\end{proof}

\subsection{Noisy power-routing guarantee}

Consider the stochastic full-magnitude route $\mathbf u=\widehat{\mathbf g}(\mathbf q)$.  Write $\widehat{\mathbf g}=\mathbf g+\boldsymbol\xi$, and construct the feasibility mask and power score from the same sampled gradient.  For $i\in G$, let
\[
    \widehat a_i=M_i|\widehat g_i|,
    \qquad
    b_i=b_{\gamma,i}(\widehat{\mathbf a}).
\]
We restrict this result to the uncapped regime $\eta\widehat a_i b_i\le s_G$.  If this condition fails, the cap can truncate the largest power scores and the deterministic comparison in Lemma~\ref{lem:power-score} need not be preserved.

The effect of amplification can be summarized by the following net margin:
\begin{equation}
\label{eq:noisy-power-net-margin}
\begin{aligned}
    \Delta_\gamma(\mathbf q)
    :={}&
    \E\!\left[
       \sum_G\sum_{i\in G}(b_i-1)
       \left(\widehat a_i^2-\frac12L_Gs_G\widehat a_i\right)
       \,\middle|\,\mathbf q\right]\\
    &-\frac1{\sqrt n}\sum_G
       \left(\E\!\left[
       \sum_{i\in G}(b_i-1)^2\widehat a_i^2
       \,\middle|\,\mathbf q\right]\right)^{1/2}
       \|\boldsymbol\sigma_G\|_2 .
\end{aligned}
\end{equation}
The first term is the additional first-order signal captured by power routing minus its curvature cost; the second accounts for selecting the routing scores from noisy gradients.  At $\gamma=0$, $b_i=1$ and $\Delta_0(\mathbf q)=0$.

\begin{theorem}[Noisy QAR-$\gamma$ trade-off]
\label{thm:noisy-qar-gamma-tradeoff}
Suppose Assumptions~\ref{ass:smooth} and~\ref{ass:stochastic-noise} hold.  Run Algorithm~\ref{alg:qar} with $\mathbf u_t=\widehat{\mathbf g}(\mathbf q_t)$ and $\mathcal A_\gamma$, $\gamma\ge0$, and assume $\eta\widehat a_{t,i}b_{t,i}\le s_G$ almost surely for every $t<T$, every group $G$, and every $i\in G$.  Let $f^*_{\mathcal Q}=\min_{\mathbf q\in\mathcal Q}f(\mathbf q)$.  Then
\begin{equation}
\label{eq:noisy-power-cumulative-tradeoff}
\begin{aligned}
    \frac1T\sum_{t=0}^{T-1}
       \E\|\mathbf g_{\mathcal Q}(\mathbf q_t)\|_2^2
    \le{}&
       \frac{2(f(\mathbf q_0)-f^*_{\mathcal Q})}{\eta T} +\underbrace{\frac14\sum_G|G|L_G^2s_G^2}
          _{\text{finite-grid floor}}+\underbrace{
          \frac1{\sqrt n}\sum_GL_Gs_G\|\boldsymbol\sigma_G\|_1
          +\frac{2\|\boldsymbol\sigma\|_2^2}{n}}
          _{\text{stochastic-noise floor}} \\
          &-\frac2T\sum_{t=0}^{T-1}
          \E\Delta_\gamma(\mathbf q_t).
\end{aligned}
\end{equation}
\end{theorem}

Because $\Delta_0=0$, the first three terms on the right-hand side form the corresponding groupwise QAR-$0$ certificate.  Power amplification gives a strictly sharper certificate whenever
\[
    \sum_{t=0}^{T-1}\E\Delta_\gamma(\mathbf q_t)>0.
\]
When this quantity is nonpositive, the comparison does not certify an advantage from amplification.

\begin{proof}
Condition on $(\mathbf q,\widehat{\mathbf g})$, and let $v_i=M_i g_i\sign(\widehat g_i)$.  Let $\mathcal U(\mathbf q,\widehat{\mathbf g})$ denote the one-step upper bound obtained with $b_i=1$.  The block majorizer \eqref{eq:power-block-majorizer}, Bernoulli expectation, and the uncapped probability $p_i=\eta\widehat a_i b_i/s_G$ give
\begin{align}
    \E_{\mathbf J}[f(\mathbf q^+)-f(\mathbf q)
       \mid\mathbf q,\widehat{\mathbf g}]
    \le{}&\mathcal U(\mathbf q,\widehat{\mathbf g})
       -\eta\sum_G\sum_{i\in G}(b_i-1)
          \left(\widehat a_i^2-\frac12L_Gs_G\widehat a_i\right)
          \notag\\
       &+\eta\sum_G\sum_{i\in G}
          |b_i-1|\widehat a_i|\xi_i|.
\label{eq:noisy-power-conditional-comparison}
\end{align}
Indeed,
\[
    |v_i-\widehat a_i|
    =M_i|\sign(\widehat g_i)|\,|g_i-\widehat g_i|
    \le|\xi_i|.
\]
Conditional Cauchy--Schwarz and Assumption~\ref{ass:stochastic-noise} imply
\[
\begin{aligned}
    \E\!\left[
       \sum_G\sum_{i\in G}|b_i-1|\widehat a_i|\xi_i|
       \,\middle|\,\mathbf q\right]
    \le
    \frac1{\sqrt n}\sum_G
       \left(\E\!\left[
       \sum_{i\in G}(b_i-1)^2\widehat a_i^2
       \,\middle|\,\mathbf q\right]\right)^{1/2}
       \|\boldsymbol\sigma_G\|_2.
\end{aligned}
\]
The coordinate estimates established in the SGD part of Theorem~\ref{thm:stochastic-qar} give
\[
\begin{aligned}
    \E[g_i\widehat g_iM_i\mid\mathbf q]
    &\ge \big((\mathbf g_{\mathcal Q}(\mathbf q))_i\big)^2
       -\frac{\sigma_i^2}{n},\\
    \E[M_i|\widehat g_i|\mid\mathbf q]
    &\le |(\mathbf g_{\mathcal Q}(\mathbf q))_i|
       +\frac{\sigma_i}{\sqrt n}.
\end{aligned}
\]
Applying these estimates to $\mathcal U$, followed by Young's inequality, yields
\[
\begin{aligned}
    \E[\mathcal U(\mathbf q,\widehat{\mathbf g})\mid\mathbf q]
    \le{}&
       -\frac\eta2\|\mathbf g_{\mathcal Q}(\mathbf q)\|_2^2
       +\frac\eta8\sum_G|G|L_G^2s_G^2\\
       &+\frac\eta{2\sqrt n}
          \sum_GL_Gs_G\|\boldsymbol\sigma_G\|_1
       +\frac{\eta\|\boldsymbol\sigma\|_2^2}{n}.
\end{aligned}
\]
Taking conditional expectation in \eqref{eq:noisy-power-conditional-comparison} and using \eqref{eq:noisy-power-net-margin} therefore gives
\[
\begin{aligned}
    \E[f(\mathbf q^+)-f(\mathbf q)\mid\mathbf q]
    \le{}&
       -\frac\eta2\|\mathbf g_{\mathcal Q}(\mathbf q)\|_2^2
       +\frac\eta8\sum_G|G|L_G^2s_G^2\\
       &+\frac\eta{2\sqrt n}
          \sum_GL_Gs_G\|\boldsymbol\sigma_G\|_1
       +\frac{\eta\|\boldsymbol\sigma\|_2^2}{n}
       -\eta\Delta_\gamma(\mathbf q).
\end{aligned}
\]
Summing over $t<T$, using $f(\mathbf q_T)\ge f^*_{\mathcal Q}$, and dividing by $\eta T/2$ proves \eqref{eq:noisy-power-cumulative-tradeoff}.
\end{proof}

\subsection{Effect of the power exponent}

\begin{lemma}[Effect of the power exponent]
\label{lem:power-exponent-tradeoff}
Fix a nonconstant group of nonnegative magnitudes $a_i$. For $\gamma>0$,
\[
    \frac{\sum_{i\in G}(b_{\gamma,i}-1)a_i^2}
         {\sum_{i\in G}(b_{\gamma,i}-1)a_i}
\]
has a positive denominator and is nondecreasing in $\gamma$; it is strictly increasing when the group has at least three distinct magnitudes. Moreover, $\max_i b_{\gamma,i}$ and $\sum_i(b_{\gamma,i}-1)^2a_i^2$ are nondecreasing in $\gamma$.
\end{lemma}

\begin{proof}
It suffices to treat positive $a_i$, zero entries follow by continuity. Let $m=|G|$ and $S_r=\sum_i a_i^r$. The ratio in the statement equals
\[
    \frac{mS_{\gamma+2}-S_2S_\gamma}
         {mS_{\gamma+1}-S_1S_\gamma}.
\]
For $0<\gamma_1<\gamma_2$, the moment kernel $S_{x+y}$ is totally positive. Indeed, Cauchy--Binet applied to $S_{x+y}=\sum_i e^{x\log a_i}e^{y\log a_i}$ expresses each minor as a sum of products of equally signed exponential Vandermonde determinants. Hence
\[
    \det\!
    \begin{pmatrix}
    S_0&S_1&S_2\\
    S_{\gamma_1}&S_{\gamma_1+1}&S_{\gamma_1+2}\\
    S_{\gamma_2}&S_{\gamma_2+1}&S_{\gamma_2+2}
    \end{pmatrix}\ge0.
\]
Subtracting $S_{\gamma_j}/m$ times the first row from row $j+1$ shows that this determinant is nonnegative exactly when the displayed ratio at $\gamma_2$ is at least its value at $\gamma_1$. The determinant is strict with at least three distinct magnitudes; with two distinct values the ratio is constant.

For the remaining claims,
\[
    \frac{\dd}{\dd\gamma}\log\max_i b_{\gamma,i}
    =\log(\max_i a_i)-\frac{\sum_i a_i^\gamma\log a_i}{S_\gamma}
    \ge0.
\]
Finally, set $Q_\gamma=\sum_i(b_{\gamma,i}-1)^2a_i^2$ and $\ell_r=S_r'/S_r$. Direct differentiation gives
\[
    \frac{Q_\gamma'}2
    =\frac{m^2S_{2\gamma+2}}{S_\gamma^2}
       (\ell_{2\gamma+2}-\ell_\gamma)
     -\frac{mS_{\gamma+2}}{S_\gamma}
       (\ell_{\gamma+2}-\ell_\gamma).
\]
The function $\ell_r$ is nondecreasing because $\ell_r'$ is the variance of $\log a_i$ under weights proportional to $a_i^r$. Also $mS_{2\gamma+2}\ge S_\gamma S_{\gamma+2}$ by the finite-sum Chebyshev inequality. The first coefficient and first parenthesized difference are therefore no smaller than the corresponding second quantities, proving $Q_\gamma'\ge0$.
\end{proof}

The ratio in Lemma~\ref{lem:power-exponent-tradeoff} is the first-order signal gained per unit of additional crossing mass. Because the power scores $b_{\gamma,i}$ depart from uniform routing only when the within-group signals are imbalanced, this ratio quantifies the signal imbalance exploited by the amplifier. For each realized group, its deterministic contribution to $\Delta_\gamma$ is positive exactly when the ratio exceeds the one-cell curvature threshold $L_Gs_G/2$. Under stochastic gradients, $\Delta_\gamma>0$ further requires these groupwise gains to dominate the selection-noise penalty in \eqref{eq:noisy-power-net-margin}. Increasing $\gamma$ favors larger signals and can strengthen the imbalance-driven gain, but it also tightens the uncapped condition $\eta\widehat a_i b_{\gamma,i}\le s_G$ and increases sensitivity to noise. Thus amplification is beneficial only when the signal imbalance is sufficiently strong relative to both curvature and stochastic noise.

\section*{AI Use Statement}

In this work, we used generative AI tools to assist with developing the proof of the feasible-gradient bound under stochastic noise and identifying candidate related work. Additionally, we used generative AI tools to refine the wording throughout the paper. We have reviewed all AI-assisted work. The authors independently checked the AI-assisted proof for mathematical correctness and consistency with the stated assumptions, manually verified suggested related work against the original sources, and reviewed and edited all AI-assisted wording to preserve the intended technical meaning and claims. We take responsibility for the final content of this work, including text, claims, and artifacts produced with the aid of generative AI.

\end{document}